\documentclass[12pt,reqno]{amsart}

\usepackage[utf8]{inputenc}
\usepackage{amssymb}
\usepackage{latexsym}
\usepackage{amsmath, amscd}
\usepackage[alphabetic]{amsrefs}

\usepackage{mathrsfs}
\usepackage{dsfont}
\usepackage{MnSymbol}
\usepackage{marvosym}

\usepackage{graphicx}
\usepackage{graphics}
\usepackage{subcaption}
\usepackage{wrapfig}

\usepackage{tikz}
\usepackage{tikz-cd}
\usepackage[all]{xy}

\usepackage{enumerate}
\usepackage{enumitem}
\usepackage{arydshln}

\usepackage{color}
\usepackage{xcolor}
\usepackage{verbatim}
\usepackage{soul}
\usepackage{ulem}
\usepackage{extarrows}

\usepackage[margin=1in]{geometry}

\usepackage{hyperref}
\definecolor{rose}{rgb}{0.82, 0.1, 0.26}
\definecolor{blu}{rgb}{0.36, 0.54, 0.66}
\definecolor{mor}{rgb}{0.55, 0.0, 0.55}

\hypersetup{
  colorlinks=true,
  linkcolor=rose,
  citecolor=mor,
  filecolor=blue
}

\usepackage{cleveref}

\theoremstyle{plain}
\newtheorem{theorem}{Theorem}[section]
\newtheorem{thm}[theorem]{Theorem}
\newtheorem{cor}[theorem]{Corollary}
\newtheorem{lem}[theorem]{Lemma}
\newtheorem{prop}[theorem]{Proposition}

\newtheorem{defi}[theorem]{Definition}

\theoremstyle{definition}

\newtheorem{ques}[theorem]{Question}

\newtheorem{rmk}[theorem]{Remark}

\theoremstyle{remark}

\theoremstyle{definition}
\newtheorem*{thm1}{Theorem 1}

\newtheorem*{coro}{Corollary}

\newcommand{\PP}{\mathds{P}}
\newcommand{\PPP}{\PP^{3}}

\newcommand{\ii}{\mathscr{I}}
\newcommand{\oo}{\mathcal{O}}

\title{The classification of ACM curves on a surface in $\PPP$}

\author{Abel Castorena}
\address{Centro de Ciencias Matemáticas, Universidad Nacional Autónoma de México, Morelia, Michoacán}
\email{abel@matmor.unam.mx}

\author{Monserrat Vite}
\address{Centro de Ciencias Matemáticas, Universidad Nacional Autónoma de México, Morelia, Michoacán}
\email{montserrat@matmor.unam.mx}
\begin{document}
\maketitle
\begin{abstract}
We classify ACM curves contained in a surface of degree $d$ in $\PPP$ in terms of weak admissible pairs. In the case of a very general smooth determinantal quartic surface, we provide a geometric description of these curves and compute their Picard classes on the surface. Finally, we present a generalization to ACM closed subvarieties of codimension $1$ on a hypersurface in $\PP^{n}$.
\end{abstract}
\section{Introduction}

Arithmetically Cohen--Macaulay curves (ACM curves, for short) play a central role in algebraic geometry due to their rich structure and favorable cohomological properties. They are defined as projective curves whose homogeneous coordinate ring is Cohen--Macaulay. ACM curves generalize the notion of complete intersections and form a natural class of objects for the study of liaison theory. For example, Apéry and Gaeta showed that a smooth curve $C$ in $\PPP$ lies in the linkage class of a complete intersection if and only if it is an ACM curve (\cite{apery}, \cite{gaeta}).\\

Moreover, ACM curves with a fixed minimal free resolution form an open smooth subset of the corresponding Hilbert scheme, whose dimension can be computed directly from the Betti numbers of the resolution \cite{ell}. The degree and genus are also determined by these numbers \cite{pesk}. In addition, the Hilbert--Burch Theorem states that a minimal set of generators of the ideal of an ACM curve is given by the minors of a matrix of size $(t+1)\times t$ \cite{har3}*{Prop. 8.7}. This theorem plays a fundamental role in the present work.\\

Regarding the classification of ACM curves, a well-known result for a smooth quadric surface $Q$ states that a curve of type $(s,t)$ on $Q$ is ACM if and only if $|s-t|\leq 1$ \cite{har2}*{III, Ex. 5.6}. In \cite{mwa}, a complete description of ACM curves on a smooth cubic surface is given, and in \cite{kwa} the author provides a classification of initialized ACM line bundles on smooth quartic surfaces in $\PPP$. A limitation of these results is that the smoothness hypothesis plays a crucial role.\\

The main purpose of this paper is to classify all ACM curves (not necessarily smooth or irreducible) contained in an arbitrary surface of degree $d$ in $\PPP$. The idea is the following: we introduce the notion of a weak determinantal surface (see Definition~\ref{defaddpair}), which generalizes the definition of a determinantal surface given in \cite{LLV}*{Def. 2}. We prove that surfaces that are not weak determinantal contain only ACM curves that are complete intersections. Furthermore, this definition allows us to determine all possible minimal free resolutions of ACM curves contained in a weak determinantal surface.\\

The main theorem of this paper is the following:
\begin{thm1}[Classification of ACM curves on a surface of degree $d$ in $\PPP$] \label{teoprinc}
Let $X=V(F)$ be a surface of degree $d$ in $\PPP$, and let $C\subseteq X$ be a curve. Then $C$ is ACM if and only if one of the following cases occurs:
\begin{enumerate}
    \item The surface $X$ is not weak determinantal, and $C$ is a complete intersection.

    \item The surface $X$ is weak determinantal of type $\{(\hat{a}_{i},\hat{b}_{i})\}_{i\in I}$ and only of these types (see Definition~\ref{defaddpair}). In this case, $C$ satisfies one of the following:
    \begin{enumerate}[label=(\roman*)]
        \item $C$ is a complete intersection.
        \item If $C$ is not a complete intersection and $F$ is a minimal generator of the ideal of $C$, then $C$ has the following minimal free resolution:
        $$0\to \bigoplus_{j=1}^{t}\oo_{\PPP}(-(b_{j}+k))\to \bigoplus_{i=1}^{t}\oo_{\PPP}(-(a_{i}+k))\oplus \oo_{\PPP}(-d)\to \ii_{C}\to 0,$$
        where $(\hat{a}_{j}=(a_{1},\ldots ,a_{t}),\hat{b}_{j}=(b_{1},\ldots ,b_{t}))$ for some $j\in I$ and some $k\in\mathds{Z}$.
        \item If $C$ is not a complete intersection and $F$ is not a minimal generator of the ideal of $C$, then $C$ has the following minimal free resolution:
        $$0\to \bigoplus_{i=1,i\neq j_{0}}^{t}\oo_{\PPP}(-d+b_{j_{0}}-b_{i})\to \bigoplus_{i=1}^{t}\oo_{\PPP}(-d+b_{j_{0}}-a_{i})\to \ii_{C}\to 0,$$
        where $(\hat{a}_{j}=(a_{1}+k,\ldots ,a_{t}+k),\hat{b}_{j}=(b_{1}+k,\ldots ,b_{j_{0}-1}+k,d,b_{j_{0}+1}+k,\ldots ,b_{t}+k))$ for some $j\in I$, $j_{0}\in\{1,\ldots ,t\}$, and some $k\in\mathds{Z}$.
    \end{enumerate}
\end{enumerate}
\end{thm1}

As an application, we compare our results with the known cases of smooth surfaces of degrees two and three. We also compute all possible weak admissible pairs of degree $4$ (see Definition~\ref{def::admissiblePair}) and determine which of them define smooth weak determinantal surfaces. In particular, we prove the following result:
\begin{coro}[Corollary \ref{Xsuave}]
If $X$ is a smooth weak determinantal surface of degree $4$, then $X$ is determinantal or of type $((0,0,1),(1,2,2))$.
\end{coro}

Since the Picard group of a very general determinantal surface is known, we provide a complete geometric description of the ACM curves contained in a very general determinantal surface.\\

\textbf{Organization:} Section~\ref{sec::1} contains the definition of a weak determinantal surface and the proof of Theorem~\ref{teoprinc}. Section~\ref{sec::2} is devoted to the study of the known cases. Section~\ref{sec::3} describes all possible weak admissible pairs of degree $4$ and, in the case of a smooth very general determinantal quartic surface $X$, provides a complete description of the ACM curves contained in $X$. Finally, Section~\ref{sec::4} presents a generalization of the main theorem to higher dimensions.\\

\textbf{Acknowledgments:} We thank César Lozano Huerta for giving us helpful feedback and the second author thanks Manuel Leal for his support with the more technical details of this writing. The first author is supported with research grant PAPIIT IN101226 ``Teoría de Brill-Noether y aplicaciones" from UNAM. Second author is a SECIHTI fellow at Centro de Ciencias Matemáticas (UNAM), México.

\section{ACM curves on a surface of degree $d$} \label{sec::1}

Let $D$ be a nonzero effective divisor on a surface $X$ of degree $d$ in $\PPP$. Then, from the exact sequence
\begin{equation} \label{sucideal}
    0 \to \oo_{\PPP}(-d)\to \ii_{D} \to \oo_{X}(-D)\to 0,
\end{equation}
the line bundle $\oo_{X}(D)$ is ACM if and only if $D$ is an ACM curve, where $\ii_{D}$ denotes the ideal sheaf of $D$ in $\PPP$. We can say more even when $D$ is not an ACM curve by relating it to linked curves.

\begin{defi}
Two equidimensional curves $D_{1}$ and $D_{2}$ in $\PP^{3}$ are \textit{directly geometrically linked} (or simply linked) by the complete intersection of two surfaces $S$ and $S^{\prime}$ if they have no components in common and $D_{1}\cup D_{2}=S\cap S^{\prime}$. In this case, we usually say that $D_{2}$ is \textit{residual} to $D_{1}$.
\end{defi}

We denote by $M_{D}:=\bigoplus_{n\in\mathds{Z}} H^{1}(\PPP,\ii_{D}(n))$ the Rao module of $D$ (see \cite{mdp2}*{I, Def. 3.1}). One relation between the Rao module and linked curves is that the Rao module $M_{D}$ is invariant under linkage, up to duals and shifts in the grading \cite{mdp2}*{III, Prop. 1.2}.\\

A relation between two curves lying in the same Picard class on a surface $X$ in $\PPP$ is given by the following proposition.

\begin{prop}\label{propclases}
Let $D_{1}$ and $D_{2}$ be two curves contained in a surface $X$ in $\PPP$. If $D_{1}$ and $D_{2}$ have the same class in $\mathrm{Pic}(X)$, then $M_{D_{1}}\cong M_{D_{2}}$. Moreover, there exists a curve $B$ such that each $D_{i}$ is linked to $B$ by the complete intersection of $X$ with a surface of degree $n$ \cite{mdp2}*{III, Def. 1.1}.
\end{prop}

\begin{proof}
Since $[D_{1}]=[D_{2}]$ in $\mathrm{Pic}(X)$, we have $\oo_{X}(-D_{1})\cong \oo_{X}(-D_{2})$. By the exact sequence \eqref{sucideal}, this yields
\[
H^{1}(\PPP, \ii_{D_{1}}(n))\cong H^{1}(X, \oo_{X}(nH-D_{1}))
\cong H^{1}(X,\oo_{X}(nH-D_{2}))
\cong H^{1}(\PPP, \ii_{D_{2}}(n)).
\]
Therefore,
\[
M_{D_{1}}=\bigoplus_{n\in\mathds{Z}} H^{1}(\PPP,\ii_{D_{1}}(n))
\cong
\bigoplus_{n\in\mathds{Z}} H^{1}(\PPP,\ii_{D_{2}}(n))
= M_{D_{2}}.
\]

For the second statement, let $H$ denote the hyperplane section class of $X$ in $\mathrm{Pic}(X)$ and set $A=[D_{1}]=[D_{2}]$. For $n\gg0$, the class $nH-A$ is very ample, hence there exists a curve $B\in |nH-A|$. Since $D_{1},D_{2}\in |A|=|nH-[B]|$, and this class corresponds to surfaces of degree $n$ containing $B$, the result follows.\\
\end{proof}

We now recall the definition of an admissible pair from \cite{LLV}*{Def. 1}, with a slight modification in the third condition.

\begin{defi}\label{def::admissiblePair}
Consider two sequences of integers $\hat{a}=(a_1,\ldots,a_t)$ and $\hat{b}=(b_1,\ldots,b_t)$ in $\mathbb{Z}^t$, with $t\geq2$, satisfying:
\begin{enumerate}[label=(\roman*)]
    \item $a_1\leq a_2\leq\cdots\leq a_t$,
    \item $b_1\leq b_2\leq\cdots\leq b_t$,
    \item $a_i<b_i$ for all $1\leq i\leq t$.
\end{enumerate}
We say that $(\hat{a},\hat{b})$ is a \textit{weak admissible pair} of degree $d:=\sum_{i=1}^{t}(b_i-a_i)$ and length $t$. Two weak admissible pairs $(\hat{a},\hat{b})$ and $(\hat{a}',\hat{b}')$ are said to be \textit{equivalent} if there exists $k\in\mathbb{Z}$ such that $a_i'=a_i+k$ and $b_i'=b_i+k$ for all $1\leq i\leq t$.
\end{defi}

Using this definition, we introduce the notion of a weak determinantal surface.

\begin{defi} \label{defaddpair}
Let $(\hat{a}, \hat{b})$ be a weak admissible pair of degree $d$ and length $t$. A surface $X=V(F)\subset\PP^3$ is called a \textit{weak determinantal surface} of type $(\hat{a}, \hat{b})$ if $F=\det(S)$ for some matrix $S=(m_{ij})$, where $m_{ij}$ is a homogeneous polynomial of degree $b_j-a_i$. In the case $a_i\geq b_j$, we assume $m_{ij}=0$.
\end{defi}

\begin{rmk}
Observe that if $a_i<b_j$ for all $i,j\in\{1,\ldots,t\}$, we recover the definition of a determinantal surface given in \cite{LLV}*{Def. 1, Def. 2}. In particular, every determinantal surface is weak determinantal.
\end{rmk}

The relevance of this class of surfaces for our work is given by the following proposition.

\begin{prop} \label{propnotdet}
Let $X=V(F)\subseteq \PPP$ be a surface of degree $d$. If there exists an ACM curve $C\subseteq X$ that is not a complete intersection, then $X$ is a weak determinantal surface of type $(\hat{a}, \hat{b})$ for some weak admissible pair $(\hat{a}, \hat{b})$ of degree $d$ and length $t$.
\end{prop}

\begin{proof}
Let $C\subseteq X$ be an ACM curve that is not a complete intersection. Two cases may occur.

\begin{itemize}
\item[Case 1:] $F$ is a minimal generator of the ideal of $C$.\\
By the Hilbert--Burch Theorem, there exists a matrix $M$ of size $(t+1)\times t$ such that $F$ is the determinant of a $t\times t$ minor of $M$. By degree considerations, we have $t\leq d$. After a change of coordinates, we may assume that the Betti numbers of the minimal free resolution of the ideal of $C$ are ordered such that $a_{0}=d$, $a_{1}\leq\cdots\leq a_{t}$, and $b_{1}\leq\cdots\leq b_{t}$. Let $m=(m_{ij})$ be the $t\times t$ submatrix of $M$ such that $\det(m)=F$ and $\deg(m_{ij})=b_j-a_i$. Setting $\hat{a}=(a_1,\ldots,a_t)$ and $\hat{b}=(b_1,\ldots,b_t)$, it follows that $X$ is weak determinantal of type $(\hat{a},\hat{b})$.

\item[Case 2:] $F$ is not a minimal generator of the ideal of $C$.\\
Let $F_1,\ldots,F_t$ be a minimal generating set of the ideal of $C$. By the Hilbert--Burch Theorem, there exists a matrix $M=(m_{ij})$ of size $t\times(t-1)$ such that each $F_i$ is the determinant of a $(t-1)\times(t-1)$ minor of $M$. Since $F$ lies in the ideal of $C$, there exist homogeneous polynomials $m_{1t},\ldots,m_{tt}\in\mathds{C}[x,y,z,w]$ such that
\[
F=\sum_{i=1}^{t}(-1)^i m_{it}F_i.
\]
Let $S$ be the matrix obtained by completing $M$ with the column $(m_{it})$. Expanding $\det(S)$ along the last column yields $\det(S)=\pm F$. The degrees satisfy $\deg(m_{ij})=b_j-a_i$, where $a_1,\ldots,a_t$ and $b_1,\ldots,b_{t-1}$ are the Betti numbers of the minimal free resolution of the ideal of $C$, ordered increasing. Setting $b_t=d$, $\hat{a}=(a_1,\ldots,a_t)$, and
\[
\hat{b}=(b_1,\ldots,b_{j_0-1},b_t,b_{j_0},\ldots,b_{t-1})
\quad\text{if } b_{j_0-1}\leq b_t<b_{j_0},
\]
we conclude that $X$ is a weak determinantal surface of type $(\hat{a},\hat{b})$.

Finally, we must verify that $m_{it}=0$ whenever $a_i=d$. If not, there exists $F_i$ of degree $d$ such that $m_{it}\neq0$, and hence
\[
F_i=\sum_{j\neq i} m_{jt}'F_j + m_{it}^{-1}F.
\]
Thus, $\{F_1,\ldots,F_{i-1},F,F_{i+1},\ldots,F_t\}$ would be a minimal generating set of the ideal of $C$, contradicting the hypothesis.
\end{itemize}
\end{proof}
\begin{rmk}
This proof generalizes the determinantal criterion given in \cite{LLV}*{Prop. 1.8}.
\end{rmk}

\begin{proof}[Proof of Theorem 1]
Item (1) is a direct consequence of Proposition~\ref{propnotdet}. For item (2), we distinguish the following cases.
\begin{itemize}
\item[Case 0:] $C$ is a complete intersection.

\item[Case 1:] $C$ is not a complete intersection and $F$ is a minimal generator of the ideal of $C$.\\
By the proof of Proposition~\ref{propnotdet}, there exists a matrix $m=(m_{ij})$ of size $t\times t$ such that $\det(m)=F$ and $\deg(m_{ij})=b_j-a_i$ for all $i,j\in\{1,\ldots,t\}$, where $a_0,a_1,\ldots,a_t,b_1,\ldots,b_t$ are the Betti numbers of the minimal free resolution of the ideal of $C$. Let $\hat{a}=(a_1,\ldots,a_t)$ and $\hat{b}=(b_1,\ldots,b_t)$. By definition, $X$ is weak determinantal of type $(\hat{a},\hat{b})$. Hence, there exists $j\in I$ such that $(\hat{a},\hat{b})$ is equivalent to $(\hat{a}_j,\hat{b}_j)$, meaning that there exists $k\in\mathds{Z}$ such that $a_i=a_{ji}+k$ and $b_i=b_{ji}+k$ for all $i\in\{1,\ldots,t\}$, and moreover
\[
a_0=\sum_{i=1}^{t}(b_i-a_i)=d.
\]

\item[Case 2:] $C$ is not a complete intersection and $F$ is not a minimal generator of the ideal of $C$.\\
Let $F_1,\ldots,F_t$ be a minimal set of generators of the ideal of $C$. By the proof of Proposition~\ref{propnotdet}, there exists a matrix $S$ such that expanding $\det(S)$ with respect to the last column yields $\det(S)=\pm F$, where $\deg(m_{ij})=b_j-a_i$, with $a_1,\ldots,a_t$ and $b_1,\ldots,b_{t-1}$ the Betti numbers of the minimal free resolution of the ideal of $C$, ordered increasing. Setting $b_t=d$, $\hat{a}=(a_1,\ldots,a_t)$, and
\[
\hat{b}=(b_1,\ldots,b_{j_0-1},b_t,b_{j_0+1},\ldots,b_{t-1})
\quad\text{if } b_{j_0-1}\leq b_t<b_{j_0+1},
\]
we conclude that $X$ is a weak determinantal surface of type $(\hat{a},\hat{b})$. Hence, there exists $j\in I$ such that $(\hat{a},\hat{b})$ is equivalent to $(\hat{a}_j,\hat{b}_j)$, that is, there exists $k\in\mathds{Z}$ such that $a_i=a_{ji}+k$ and $b_i=b_{ji}+k$ for all $i\in\{1,\ldots,t\}$. In particular, we have $d=b_{jj_0}+k$, thus $k=d-b_{jj_0}$.
\end{itemize}
\end{proof}

\begin{rmk}
\begin{enumerate}
\item In item (2), case (ii), if $k=d-b_j$ for some $j$, then $F$ ceases to be a minimal generator of the ideal of $C$.
\item While the definition of admissible pairs in \cite{LLV} yields only finitely many admissible pairs, in our setting there are infinitely many weak admissible pairs.
\end{enumerate}
\end{rmk}

To conclude this section, we associate a matrix to each admissible pair. For this purpose, we define the following function:
\[
\delta:\mathds{Z}\times\mathds{Z}\longrightarrow\mathds{N}, \qquad
(a,b)\longmapsto
\begin{cases}
b-a, & \text{if } b-a>0,\\
0, & \text{if } b-a\leq 0.
\end{cases}
\]

\begin{defi}
Let $(\hat{a},\hat{b})$ be a weak admissible pair. We define the degree matrix associated to $(\hat{a},\hat{b})$ by
\[
M_{(\hat{a},\hat{b})}:=
\begin{pmatrix}
\delta(a_1,b_1) & \delta(a_1,b_2) & \cdots & \delta(a_1,b_t)\\
\delta(a_2,b_1) & \delta(a_2,b_2) & \cdots & \delta(a_2,b_t)\\
\vdots & \vdots & \ddots & \vdots\\
\delta(a_t,b_1) & \delta(a_t,b_2) & \cdots & \delta(a_t,b_t)
\end{pmatrix}.
\]
\end{defi}

Note that if two admissible pairs are equivalent, then their degree matrices coincide. On the other hand, given an admissible pair $(\hat{a},\hat{b})$, there may exist another admissible pair $(\hat{a}',\hat{b}')$ such that
\[
M_{(\hat{a}',\hat{b}')} = M_{(\hat{a},\hat{b})}^{a},
\]
where $M^{a}$ denotes the transpose of a matrix with respect to its antidiagonal. In this case, we say that $(\hat{a},\hat{b})$ and $(\hat{a}',\hat{b}')$ are \textit{dual}.\\

Although there are infinitely many weak admissible pairs, most of them define the same collection of weak determinantal surfaces. In order to identify all such pairs, we need the following result.

\begin{lem}
Let $M_{(\hat{a},\hat{b})}=(m_{ij})$ be the degree matrix of an admissible pair $(\hat{a},\hat{b})$ of degree $d$ and length $t$. If $m_{ij}>0$, then $m_{ji}<d$.
\end{lem}

\begin{proof}
Suppose that $0<m_{ij}=\delta(a_i,b_j)=b_j-a_i$ and that $m_{ji}\geq d$. Then $m_{ji}=b_i-a_j=d+n$ for some $n\in\mathds{N}$. On the other hand,
\[
d=\sum_{k=1}^{t}(b_k-a_k)
=\sum_{k\neq i,j}(b_k-a_k)+b_j-a_i+d+n.
\]
Since $b_k-a_k>0$ for all $k$, this implies $b_j-a_i<0$, a contradiction. Therefore, $m_{ji}<d$.\\
\end{proof}

\begin{defi}\label{tipoadpar}
Two admissible pairs $(\hat{a},\hat{b})$ and $(\hat{a}',\hat{b}')$ are said to be of the same \textit{kind} if they have the same degree and length, and if their degree matrices $M_{(\hat{a},\hat{b})}=(m_{ij})$ and $M_{(\hat{a}',\hat{b}')}=(m_{ij}')$ satisfy $m_{ij}=m_{ij}'$ for all entries with $0<m_{ij}<d$.
\end{defi}

\begin{rmk}\label{remadmpar}
\begin{enumerate}
\item If $X$ is a weak determinantal surface of type $(\hat{a},\hat{b})$, then $X$ is also weak determinantal of type $(\hat{a},\hat{b})^{a}$, where $(\hat{a},\hat{b})^{a}$ is dual to $(\hat{a},\hat{b})$.

\item Let $(\hat{a},\hat{b})$ and $(\hat{a}',\hat{b}')$ be weak admissible pairs of the same kind. With the notation above, observe that if $m_{ij}=0$, then $i>j$. Since the pairs are the same kind, we have either $m_{ij}'=0$ or $m_{ji}'=0$, but as $i>j$, it follows that $m_{ij}'=0$. Hence, $m_{ij}=0$ if and only if $m_{ij}'=0$. This implies that the kind of a weak admissible pair determines the positions of the zeros in the degree matrix. Consequently, there are only finitely many kinds. Moreover, if two weak admissible pairs are the same kind, then a surface $X$ is weak determinantal of one type if and only if it is weak determinantal of the other.

\item One way to determine all admissible pairs associated to a given surface $X$ is to consider curves obtained as complete intersections of $X$ with surfaces of low degree, and then link these curves by complete intersections of $X$ with other surfaces containing them. Computing the resolution of the linked curve yields an admissible pair (this step follows from the fact that we know the minimal resolution of a complete intersection and \cite{mdp2}*{III Prop.1.4}); taking its dual and all admissible pairs with degree matrix of the same type, and repeating this process finitely many times, produces all admissible pairs defining $X$ as a weak determinantal surface.
\end{enumerate}
\end{rmk}

\begin{ques}
    From all above, a natural question is the following: Given a degree $d$ surface in $\PPP$, how can one determine that the set of all weak admissible pairs for X has been fully characterized?
\end{ques}

\section{The known cases} \label{sec::2}

In this section, we analyze the cases of surfaces of degrees $2$ and $3$. These cases have already been completely understood when the surface is smooth. For the degree $2$ case, we refer to \cite{har2}*{III, Ex. 5.6}. For the degree $3$ case, we refer to the work of Masayuki Watanabe in \cite{mwa}*{Prop. 4.3}.

\subsection{Degree two}

The possible weak admissible pairs are $((1,1),(2,2))$ and $((1,1+n),(2,2+n))$ for $n\in\mathds{N}^{*}$. It is well known that all smooth quadric surfaces are determinantal, which means that all smooth quadrics are of type $((1,1),(2,2))$. On the other hand, if a quadric surface is of type $((1,1+n),(2,2+n))$ for some $n\in\mathds{N}^{*}$, then its degree matrix is of the form
\[
M_{((1,1+n),(2,2+n))}=
\begin{pmatrix}
1 & n+1\\
0 & 1
\end{pmatrix}.
\]
This implies that a weak determinantal surface of this type is reducible and, in particular, singular. Therefore, a smooth quadric surface cannot be of type $((1,1+n),(2,2+n))$.

\begin{cor}
Let $Q$ be a smooth quadric surface and let $C\subseteq Q$ be a curve. Then $C$ is an ACM curve if and only if one of the following conditions holds:
\begin{enumerate}
\item $C$ is a complete intersection, or
\item $C$ has the following minimal free resolution:
\[
0\to \oo_{\PPP}(-(2+k))^{2}\to \oo_{\PPP}(-(1+k))^{2}\oplus \oo_{\PPP}(-2)\to \ii_{C}\to 0
\]
for $k\geq 1$.
\end{enumerate}
\end{cor}

In both cases, the curves are of type $(s,t)$ on $Q$ with $|s-t|\leq 1$ (see \cite{har2}*{III, Ex. 5.6}).\\

On the other hand, observe that all weak admissible pairs of the form $((1,1+n),(2,2+n))$ have the same type for all $n\geq 1$. Hence, we obtain the following result.

\begin{cor}\label{singularcuadric}
Let $Q$ be a quadric surface of type $((1,2),(2,3))$ and let $C\subseteq Q$ be a curve. Then $C$ is an ACM curve if and only if one of the following conditions holds:
\begin{enumerate}
\item $C$ is a complete intersection, or
\item $C$ has the following minimal free resolution:
\[
0\to \oo_{\PPP}(-(1+k))\oplus \oo_{\PPP}(-(n+1+k))\to
\oo_{\PPP}(-(2+k))\oplus \oo_{\PPP}(-(2+k+n))\oplus \oo_{\PPP}(-2)
\to \ii_{C}\to 0
\]
for $n,k\geq 1$.
\end{enumerate}
\end{cor}

\subsection{Degree three}

In this case, the classification of all weak admissible pairs is more involved. We therefore organize the admissible pairs according to their type.

 $$ M_{1}=\left( \begin{array}{cc}
1 & 2\\
1 & 2
\end{array}
    \right)  \text{ the unique weak admissible pair with this matrix is } ((1,1),(2,3)).$$
 $$ M_{2}=\left( \begin{array}{cc}
2 & 2\\
1 & 1
\end{array}
    \right)  \text{ the unique weak admissible pair with this matrix is } ((1,2),(3,3)). $$
    $$ M_{3}(n)=\left( \begin{array}{cc}
2 & n\\
0 & 1
\end{array}
    \right)  \text{ the weak admissible pairs with this matrix are } ((1,n),(3,n+1)) \text{ with } n\geq 4. $$
    $$ M_{4}(n)=\left( \begin{array}{cc}
1 & n\\
0 & 2
\end{array}
    \right)  \text{ the weak admissible pairs with this matrix are } ((1,n),(2,2+n)) \text{ with } n\geq 2. $$
 $$ M_{5}=\left( \begin{array}{ccc}
1 & 1 & 1\\
1 & 1 & 1 \\
1 & 1 & 1
\end{array}
    \right) \text{ the unique weak admissible pair with this matrix is } ((1,1,1),(2,2,2)).$$
$$ M_{6}(n)=\left( \begin{array}{ccc}
1 & 1 & n\\
1 & 1 & n \\
0 & 0 & 1
\end{array}
    \right) \begin{array}{c}
\text{ the weak admissible pairs with this matrix are } ((1,1,n),(2,2,n+1)) \\\text{ with } n\geq 2.\end{array} $$ 
$$ M_{7}(n)=\left( \begin{array}{ccc}
1 & n & n\\
0 & 1 & 1 \\
0 & 1 & 1
\end{array}
    \right) \begin{array}{c}\text{ the weak admissible pairs with this matrix are } \\((1,n,n),(2,n+1,n+1)) \text{ with } n\geq 2.\end{array}$$

 $$ M_{8}(m,n)=\left( \begin{array}{ccc}
1 & m+1 & n+1\\
0 & 1 & n-m+1 \\
0 & 0 & 1
\end{array}
    \right) \begin{array}{c}\text{ the weak admissible pairs with this matrix are } \\ 
    ((0,m, n),(1, m+1, n+1)) \\
    \text{ with } 1\leq m<n.\end{array}$$

Observe that $M_{1}^{a}=M_{2}$, $M_{3}(n)^{a}=M_{4}(n)$, and $M_{6}(n)^{a}=M_{7}(n)$. If a weak determinantal surface $X$ has degree matrix of type $M_{3}(n)$, $M_{4}(n)$, $M_{6}(n)$, $M_{7}(n)$, or $M_{8}(m,n)$, then $X$ is necessarily reducible. Therefore, a smooth surface cannot be of any of these types. However, there exist smooth cubic weak determinantal surfaces with degree matrices $M_{1}$, $M_{2}$, or $M_{5}$, and the general element of each of these types is smooth.

\begin{cor}
Let $X$ be a smooth cubic surface and let $C\subseteq X$ be a curve. Then $C$ is an ACM curve if and only if one of the following cases occurs:
\begin{enumerate}
\item $X$ is a general weak determinantal surface of type $((1,1),(2,3))$ and $C$ is one of the following:
\begin{enumerate}[label=(\roman*)]
\item $C$ is a complete intersection.
\item $C$ has one of the following minimal free resolutions:
\begin{enumerate}[label=\Alph*]
\item For $k\geq 0$,
\[
0\to \oo_{\PPP}(-(3+k))\oplus \oo_{\PPP}(-(2+k))
\to \oo_{\PPP}(-(1+k))^{2}\oplus \mathcal{O}_{\PPP}(-3)
\to \ii_{C}\to 0.
\]
\item For $k\geq 0$,
\[
0\to \oo_{\PPP}(-(3+k))^{2}
\to \oo_{\PPP}(-(1+k))\oplus \oo_{\PPP}(-(2+k))\oplus \mathcal{O}_{\PPP}(-3)
\to \ii_{C}\to 0.
\]
\end{enumerate}
\end{enumerate}

\item $X$ is a general determinantal surface of type $((1,1,1),(2,2,2))$ and $C$ is one of the following:
\begin{enumerate}[label=(\roman*)]
\item $C$ is a complete intersection.
\item $C$ has the following minimal free resolution:
\[
0\to \oo_{\PPP}(-(2+k))^{3}
\to \oo_{\PPP}(-(1+k))^{3}\oplus \mathcal{O}_{\PPP}(-3)
\to \ii_{C}\to 0
\]
for $k\geq 1$.
\end{enumerate}
\end{enumerate}
\end{cor}

These resolutions coincide with those obtained in \cite{mwa}*{Prop. 4.8}. Finally, the resolutions in the singular case can be obtained analogously to the quadric case described in Corollary~\ref{singularcuadric}.

\section{Degree four} \label{sec::3}

In this section, we focus on the case of ACM curves contained in a surface of degree four. Before computing all possible weak admissible pairs, we first state some useful properties that are direct consequences of the definition of a weak admissible pair.

\begin{rmk} \label{Obsparejas}
Let $(\hat{a}, \hat{b})$ be a weak admissible pair of degree $d$ and length $t$.
\begin{enumerate}
    \item $b_{j}-a_{i}>0$ for all $i\leq j$.
    \item If $b_{j}-a_{i}=c$, then $b_{j-k}-a_{i+l}\leq c$ for all $0\leq k\leq j-1$ and $0\leq l \leq t-i$.
    \item If $b_{j}-a_{i}\geq d$, then $b_{i}-a_{j}\leq 0$.
    \item Finding all weak admissible pairs of degree $d$ reduces to computing all possible degree matrices with trace $d$.
\end{enumerate}
\end{rmk}

\subsection{Weak admissible pairs of degree four}

Our goal is to determine all weak admissible pairs of degree four. We proceed by analyzing cases according to the length $t$. For each $t$, we follow these steps:
\begin{enumerate}[label=(\roman*)]
    \item Impose the trace condition $\sum_{i=1}^t m_{ii}=4$ with $m_{ii}>0$.
    \item Assume $a_1=0$ (normalization).
    \item Study the zero patterns in the lower triangular part of the degree matrix $M$.
\end{enumerate}

\vskip2mm
\noindent We now analyze each case according to the length $t$.

\subsubsection{\textbf{Case $t=2$}}

The degree matrix has the form $M=(m_{ij})$, with $m_{11}+m_{22}=4,\, m_{12}>0$ and $m_{21}\geq 0.$

If $m_{21}>0$, then the pair is an admissible pair in the sense of \cite{LLV}, already classified as follows:
$$ M_{1}=\left( \begin{array}{cc}
1 & 3\\
1 & 3
\end{array}
    \right)  ,\quad ((1,1),(2,4)).$$
  $$ M_{2}=\left( \begin{array}{cc}
3 & 3\\
1 & 1
\end{array}
    \right)  ,\quad ((1,3),(4,4)). $$  

    $$ M_{3}=\left( \begin{array}{cc}
2 & 3\\
1 & 2
\end{array}
    \right)  ,\quad ((1,2),(3,4)).$$
        $$ M_{4}=\left( \begin{array}{cc}
2 & 2\\
2 & 2
\end{array}
    \right)  ,\quad ((1,1),(3,3)). $$
We therefore assume $m_{21}=0$ in order to obtain new weak admissible pairs.

\begin{itemize}
\item \textbf{$m_{11}=m_{22}=2$.}  
From $m_{11}=b_1=2$ and the inequality $4=m_{11}+m_{22}\leq m_{12}=b_2$, we obtain $b_2=4+n$ for $n\in\mathds{N}$, hence $a_2=2+n$. This yields the weak admissible pair
\[
M_5(n)=\begin{pmatrix}2&4+n\\0&2\end{pmatrix},\quad ((0,2+n),(2,4+n)),\ n\geq 0.
\]

\item \textbf{$m_{11}=3,\ m_{22}=1$.}  
Then $b_1=3$ and $4\leq b_2$, so $b_2=4+n$ and $a_2=3+n$. This gives
\[
M_6(n)=\begin{pmatrix}3&4+n\\0&1\end{pmatrix},\quad ((0,3+n),(3,4+n)),\ n\geq 0.
\]

\item \textbf{$m_{11}=1,\ m_{22}=3$.}  
This case is obtained by anti-transposition of $M_6(n)$, yielding
\[
M_7(n)=\begin{pmatrix}1&4+n\\0&3\end{pmatrix},\quad ((0,1+n),(1,4+n)),\ n\geq 0.
\]
\end{itemize}

\subsubsection{\textbf{Case $t=3$}}

The degree matrix has the form $M=(m_{ij})$, with $m_{11}+m_{22}+m_{33}=4.$

If all $m_{ij}>0$ for $j<i$, then if the pair is admissible:
\[
M_{8}=\begin{pmatrix}2&2&2\\1&1&1\\1&1&1\end{pmatrix},
\quad ((1,2,2),(3,3,3)),
\]
\[
M_{9}=\begin{pmatrix}1&1&2\\1&1&2\\1&1&2\end{pmatrix},
\quad ((1,1,1),(2,2,3)).
\]

To find new weak admissible pairs, we analyze matrices with zeros below the diagonal.

\begin{itemize}
    \item \textit{\textbf{Configuration $(2, 1, 1)$}}\\
Assume $(m_{11}, m_{22}, m_{33}) = (2, 1, 1)$ with $b_1 = 2$.

\begin{itemize}
    \item \textbf{$m_{31} = m_{32} = 0,\ m_{21} > 0$}: Solving the conditions yields the following weak admissible pair:
    \[
    M_{10}(n) = \begin{pmatrix}2 & 2 & 3 + n\\ 1 & 1 & 2 + n\\ 0 & 0 & 1\end{pmatrix},\quad ((0, 1, 2 + n),(2, 2, 3 + n)), n\geq 0.
    \]
    
    \item \textbf{$m_{31} = m_{21} = 0,\ m_{32} > 0$}: This case leads to the weak admissible pair:
    \[
    M_{11}(n) = \begin{pmatrix}2 & 3 + n & 3 + n\\ 0 & 1 & 1\\ 0 & 1 & 1\end{pmatrix},\quad ((0, 2 + n, 2 + n),(2, 3 + n, 3 + n)), n\geq 0.
    \]
    
    \item \textbf{$m_{31} = m_{32} = m_{21} = 0$}: Two distinct families of weak admissible pairs arise:
    \begin{align*}
    M_{12}(n) &= \begin{pmatrix}2 & 3 + n & 4 + n\\ 0 & 1 & 2 \\ 0 & 0 & 1\end{pmatrix},\quad ((0, 2 + n, 3 + n), (2, 3 + n, 4 + n)), n\geq 0\\
    M_{13}(m, n) &= \begin{pmatrix}2 & 3 + m & 3 + n\\ 0 & 1 & 1 + n - m\\ 0 & 0 & 1\end{pmatrix},\quad ((0, 2 + m, 2 + n), (2, 3 + m, 3 + n)),\ 0\leq m < n.
    \end{align*}
\end{itemize}

\item The diagonal distribution $(1, 1, 2)$ is obtained by anti-transposition of the previous cases, yielding additional families of weak admissible pairs:

$$
M_{14}(n)=\left( \begin{array}{ccc}
1 & 1 & 3+n\\
1 & 1 & 3+n \\
0 & 0 & 2
\end{array}
\right),
\quad
((0,0,1+n),(1, 1, 3+n)),\ n\geq 0.
$$ 

$$
M_{15}(n)=\left( \begin{array}{ccc}
1 & 2+n & 3+n\\
0 & 1 & 2 \\
0 & 1 & 2
\end{array}
\right),
\quad 
((0,1+n, 1+n),(1, 2+n, 3+n)),\ n\geq 0.
$$

$$
M_{16}(n)=\left( \begin{array}{ccc}
1 & 2 & 4+n\\
0 & 1 & 3+n \\
0 & 0 & 2
\end{array}
\right),
\quad
((0,1, 2+n),(1, 2, 4+n)),\ n\geq 0.
$$

$$
M_{17}(m,n)=\left( \begin{array}{ccc}
1 & 1+m & 3+n\\
0 & 1 & 3+n-m \\
0 & 0 & 2
\end{array}
\right),
\quad
((0, m, 1+n),(1, 1+m, 3+n)),\ 1 \leq m<n.
$$

\item \textit{\textbf{Configuration $(1, 2, 1)$}} \\
Assume $(m_{11}, m_{22}, m_{33}) = (1, 2, 1)$ with $b_1 = 2$.

\begin{itemize}
    \item \textbf{$m_{31}= 0,\ m_{32},m_{21} > 0$}: Solving the conditions yields the weak admissible pair:
    \[
    M_{18} = \begin{pmatrix}1 & 2 & 2\\
1 & 2 & 2 \\
0 & 1 & 1 \end{pmatrix},\quad ((0,0,1),(1,2, 2)).
    \]

    \item \textbf{$m_{31} = m_{32} = 0,\, m_{21}>0$}: Exactly one zero pattern produces weak admissible pairs:
    \[
    M_{19}(n) = \begin{pmatrix}1 & 2 & 3+n\\
1 & 2 & 3+n \\
0 & 0 & 1\end{pmatrix},\quad ((0,0, 2+n),(1, 2, 3+n)),\ n\geq 0.
    \]

    \item \textbf{$m_{31} = m_{21} = 0,\, m_{32}>0$}: This case is obtained by anti-transposition of the previous one, yielding additional families of weak admissible pairs:
    \[
    M_{20}(n) = \begin{pmatrix}1 & 3+n & 3+n\\
0 & 2 & 2 \\
0 & 1 & 1
\end{pmatrix},\quad ((0,1+n, 2+n),(1, 3+n, 3+n)),\ n\geq 0.
    \]

    \item \textbf{$m_{31} = m_{32} = m_{21} = 0$}: In this case, we obtain a family of weak admissible pairs:
    \begin{align*}
    M_{21}(n) &= \begin{pmatrix}1 & 3+m & 4+n\\
0 & 2 & 3+n-m \\
0 & 0 & 1\end{pmatrix},\quad ((0,1+m, 3+n),(1, 3+m, 4+n)),\ n\geq 0.
    \end{align*}
\end{itemize}
\end{itemize}

\subsubsection{\textbf{Case $t=4$}}

The degree matrix has the form $M=(m_{ij})$, with $\sum_{i=1}^4 m_{ii}=4.$

Since $m_{ii}>0$ for all $i$, the trace condition forces
\[
m_{11}=m_{22}=m_{33}=m_{44}=1.
\]

\noindent This follows because the minimal sum of four positive integers is $4$, which is achieved uniquely when each of them equals $1$.

\vskip2mm

\noindent If all $m_{ij} > 0$ for $j < i$, we obtain the unique admissible matrix:
$$
M_{22}=\left( \begin{array}{cccc}
1 & 1 & 1 & 1\\
1 & 1 & 1 & 1 \\
1 & 1 & 1 & 1 \\
1 & 1 & 1 & 1 
\end{array}
\right),\quad ((1, 1, 1, 1),(2, 2, 2, 2)).
$$

To find weak admissible pairs, we examine matrices with zeros below the diagonal:

\begin{itemize}
    \item \textbf{three zeros}: Two distinct zero patterns yield weak admissible pairs:
    \[
    M_{23}(n) = \begin{pmatrix}
  1 & 1 & 1 & 2+n\\
1 & 1 & 1 & 2+n \\
1 & 1 & 1 & 2+n \\
0 & 0 & 0 & 1
    \end{pmatrix},\quad 
    ((0, 0, 0, 1 + n), (1, 1, 1, 2 + n)),\ n\geq 0.
    \]
    \[
    M_{24}(n) = \begin{pmatrix}
   1 & 2+n & 2+n & 2+n\\
0 & 1 & 1 & 1 \\
0 & 1 & 1 & 1 \\
0 & 1 & 1 & 1
    \end{pmatrix},\quad ((0, 1 + n, 1 + n, 1 + n), (1, 2 + n, 2 + n, 2 + n)),\ n\geq 0.
    \]
    
    \item \textbf{four zeros}: Exactly one zero pattern produces weak admissible pairs:
    \[
    M_{25}(n) = \begin{pmatrix}
   1 & 1 & 2+n & 2+n \\
1 & 1 & 2+n & 2+n \\
0 & 0 & 1 & 1 \\
0 & 0 & 1 & 1 
    \end{pmatrix},\quad ((0, 0, 1 + n, 1 + n), (1, 1, 2 + n, 2 + n)),\ n\geq 0.
    \]
    
    \item \textbf{five zeros}: Two families of weak admissible pairs arise:
    \begin{align*}
    M_{26}(m, n) &= \begin{pmatrix}
   1 & 1 & 2+m & 3+n\\
1 & 1 & 2+m & 3+n \\
0 & 0  & 1 & 2+n-m \\ 
0 & 0 & 0 & 1
    \end{pmatrix},\quad ((0, 0, 1 + m, 2 + n), (1, 1, 2 + m, 3 + n)),\ 0\leq m \leq n\\
    M_{27}(m, n) &= \begin{pmatrix}
    1 & 2+m & 3+n & 3+n\\
0 & 1 & 2+n-m & 2+n-m \\
0 & 0  & 1 & 1 \\ 
0 & 0 & 1 & 1
    \end{pmatrix},
    \begin{array}{c}
     ((0, 1 + m, 2 + n, 2 + n), (1, 2 + m, 3 + n, 3 + n))\\
    0\leq m \leq n
    \end{array}
    \end{align*}
    
    \item \textbf{six zeros}: The maximal zero pattern yields the most general family of weak admissible pairs:
    \[
    M_{28}(k, m, n) = \begin{pmatrix}
 1 &  2+k & 3+m & 4+n\\
0 & 1 & 2+m-k & 3+n-k \\
0 & 0  & 1 & 2+n-m \\ 
0 & 0 & 0 & 
    \end{pmatrix},
    \begin{array}{c}
 ((0, 1 + k, 2 + m, 3 + n),
    (1, 2 + k, 3 + m, 4 + n)), \\
    0\leq k \leq m \leq n
    \end{array}
    \]
\end{itemize}

\noindent From all above, we have obtained 20 infinite families of weak admissible pairs.
This classification provides an exhaustive list of all weak admissible pairs ordered by their length and corresponding degree matrix families.\\

Observe that $M_{1}^{a}=M_{2}, M_{3}^{a}=M_{3}, M_{4}^{a}=M_{4}, M_{5}^{a}(n)=M_{5}(n), M_{6}^{a}(n)=M_{7}(n), M_{8}^{a}=M_{9}, M_{10}^{a}(n)=M_{15}(n),  M_{11}^{a}(n)=M_{14}(n), M_{12}^{a}(n)=M_{16}(n), M_{13}^{a}(m,n)=M_{17}(n-m,n), M_{18}^{a}=M_{18}, M_{19}^{a}(n)=M_{20}(n),$ 
$ M_{21}^{a}(m,n)=M_{21}(n-m,n), M_{22}^{a}=M_{22}, M_{23}^{a}(n)=M_{24}(n), M_{25}^{a}(n)=M_{25}(n), M_{26}^{a}(m,n)=M_{27}(n-m,n),$
 and $M_{28}^{a}(k,m,n)=M_{28}(n-m,n-k,n)$ for all $k,m,n\in{\mathds{N}^{*}}$ with $k<m<n$.

\begin{cor} \label{Xsuave}
\begin{enumerate}
    \item The surfaces $X$ weak determinantal of type $(\hat{a},\hat{b})$ with degree matrix as $M_{5}(n),$
    $ M_{6}(n), M_{7}(n), M_{10}(n), M_{11}(n), M_{12}(n), M_{14}(m,n), M_{14}(n),$\\$M_{15}(n), M_{16}(n), M_{17}(m,n), M_{19}(n), M_{20}(n), M_{21}(n), M_{23}(n), M_{24}(n), M_{25}(n),$\\ $ M_{26}(m,n),
     M_{27}(m,n)$ or $M_{28}(k,m,n)$ are reducible.

 \item If $X$ is a smooth weak determinantal surface of degree 4, then $X$ is determinantal or of type $((0,0,1),(1,2,2))$. 
\end{enumerate}
\end{cor}
\begin{proof}
Let $(\hat{a},\hat{b})$ be a weak admissible pair of degree $4$ and length $t$, and let $M=M_{(\hat{a},\hat{b})}=(m_{ij})$ be its degree matrix. Let $X$ be a weak determinantal surface of type $(\hat{a},\hat{b})$. By definition, there exists a matrix $N=(n_{ij})$ such that each $n_{ij}$ is a homogeneous polynomial of degree $m_{ij}$, with $n_{ij}=0$ if $m_{ij}=0$, and $X=V(\det(N))$.

Item (1) follows directly from Remark~\ref{Obsparejas}. If there exists $i\in\{1,\ldots,t\}$ such that $m_{(i+1)i}=0$, then $n_{(i+1+k)(i-l)}=0$ for all $0\leq k\leq t-i-1$ and $0\leq l\leq i-1$. Consequently,
\[
\det(N)=\det(N_1)\cdot\det(N_2),
\]
where $N_1=(n_{jl})_{j,l=1}^{i}$ and $N_2=(n_{pq})_{p,q=i+1}^{t}$, and therefore $X$ is reducible.

For item (2), it suffices to show that for any admissible pair, as well as for the pair $((0,0,1),(1,2,2))$, there exists a smooth element. In the admissible case, this follows from \cite{LLV}*{Prop. 1.1}. The remaining case can be verified directly using \texttt{Macaulay2}.\\
\end{proof}

\subsection{Determinantal surfaces}

By Corollary~\ref{Xsuave}, any smooth weak determinantal surface that is not of type
$((0,0,1),(1,2,2))$ is a determinantal surface. That is, $X$ is weak determinantal of
type $(\hat{a},\hat{b})$ satisfying $a_i<b_j$ for all $i,j$. These surfaces give rise
to five irreducible components of the Noether--Lefschetz locus
\[
NL(4)\subseteq |\mathcal{O}_{\mathbb{P}^3}(4)|,
\]
all of codimension $1$.

\begin{thm}[\cite{LLV}]\label{thm::quarticDivisors}
The family of determinantal quartic surfaces consists of five prime divisors
$\mathcal{F}_1,\ldots,\mathcal{F}_5\subseteq|\oo_{\PP^3}(4)|$, each corresponding to an
admissible pair $(\hat{a},\hat{b})$ of degree $4$, each divisor $\mathcal{F}_i$ coincides with a classical family of quartic
surfaces, namely those containing the following curves:
\begin{enumerate}
    \item[$\mathcal{F}_1:$] a curve of degree $6$ and genus $3$ with resolution
    \[
    0\to\oo_{\PP^3}(-4)^3\to\oo_{\PP^3}(-3)^4;
    \]
    \item[$\mathcal{F}_2:$] a twisted cubic;
    \item[$\mathcal{F}_3:$] a complete intersection of two quadric surfaces;
    \item[$\mathcal{F}_4:$] a line;
    \item[$\mathcal{F}_5:$] a conic.
\end{enumerate}
\end{thm}
All these cases are particularly interesting, since for a very general smooth element
of each divisor the Picard group has rank two. More precisely, let
$X_i\in\mathcal{F}_i$ ($1\leq i\leq5$) be a very general smooth surface and
$C_i\subseteq X_i$ the corresponding curve listed above, with degree $d_i$ and genus
$g_i$. Then $Pic(X_i)$ is generated by the hyperplane section $H$ and the class $[C_i]$,
and the intersection matrix with respect to this basis is
\[
A_i=
\begin{pmatrix}
4 & d_i\\
d_i & 2g_i-2
\end{pmatrix}.
\]

These properties motivate us to give a complete description of the ACM curves on a
very general element of each component.

\subsubsection{\textbf{Quartic surfaces containing a line}}

Let $X\in\mathcal{F}_4$ be a very general smooth quartic surface. Denote by $L=[C_4]$
the class of the line contained in $X$. Then $Pic(X)$ is generated by $H$ and $L$ with
\[
H^2=4,\quad H\cdot L=1,\quad L^2=-2.
\]

\begin{prop}\label{prop1}
Let $X=V(F)\in\mathcal{F}_4$ be a very general smooth quartic surface and let
$D\subseteq X$ be a curve such that $D\not\sim H$. Write $[D]=aH+bL$ in $Pic(X)$. Then
$D$ is an ACM curve if and only if one of the following cases occurs:
\begin{enumerate}[label=(\roman*)]
    \item If $|D-H|=\emptyset$:
    \begin{enumerate}
        \item[(a)] $[D]=L$, and the unique curve in this class is the line $C_4$.
    \end{enumerate}

    \item If $F$ is a minimal generator of $\ii_D$ and $|D-H|\neq\emptyset$:
    \begin{itemize}
        \item[(a)] $[D]=kH+L$ with $k\geq3$. These curves have degree $4k+1$, genus
        $2k^2+k$, and minimal free resolution
        \[
        0\to \oo_{\PPP}(-(4+k))\oplus \oo_{\PPP}(-(2+k))
        \to \oo_{\PPP}(-(1+k))^2\oplus \oo_{\PPP}(-4)
        \to \ii_D\to0.
        \]
        \item[(b)] $[D]=(k+1)H-L$ with $k\geq3$. These curves have degree $4k+3$, genus
        $2k^2+3k+1$, and minimal free resolution
        \[
        0\to \oo_{\PPP}(-(4+k))^2
        \to \oo_{\PPP}(-(1+k))\oplus \oo_{\PPP}(-(3+k))\oplus \oo_{\PPP}(-4)
        \to \ii_D\to0.
        \]
    \end{itemize}

    \item If $D$ is a plane curve and $|D-H|\neq\emptyset$:
    \begin{itemize}
        \item[(a)] $[D]=H-L$, and $D$ is a plane cubic residual to $C_4$ in the complete
        intersection of $X$ with a plane containing $C_4$.
    \end{itemize}

    \item If $D$ is contained in a quadric surface and $|D-H|\neq\emptyset$:
    \begin{itemize}
        \item[(a)] $[D]=2H$, a curve of degree $8$ and genus $9$;
        \item[(b)] $[D]=2H-L$, residual to $C_4$;
        \item[(c)] $[D]=H+L$, residual to a plane cubic contained in $X$.
    \end{itemize}

    \item If $D$ is contained in a cubic surface and $|D-H|\neq\emptyset$:
    \begin{itemize}
        \item[(a)] $[D]=3H$, a curve of degree $12$ and genus $19$;
        \item[(b)] $[D]=3H-L$, residual to $C_4$;
        \item[(c)] $[D]=2H+L$, residual to a plane cubic.
    \end{itemize}

    \item $D$ is a complete intersection of $X$ with a surface of degree $d\geq4$; in
    this case $[D]=dH$.
\end{enumerate}
\end{prop}
\begin{proof}
For item (i), we follow the conditions of \cite{kwa}*{Thm.1.1}.

\begin{itemize}
    \item $D^{2}=-2$ and $1\leq D.H\leq 3$.\\
    From the second condition we obtain
    $1\leq 4a+b\leq 3$, hence $1-4a\leq b\leq 3-4a$.
    Since $b\in\mathbb{Z}$, we consider the following cases:
    \begin{itemize}
        \item[$b=1-4a$] $ $\\
        From the first equation we have
        \begin{align*}
            -2=D^{2}
            &=4a^{2}+2ab-2b^{2}\\
            &=4a^{2}+2a(1-4a)-2(1-4a)^{2}\\
            &=-36a^{2}+18a-2.
        \end{align*}
        Hence the only integer solution is $a=0$, $b=1$, and therefore $D=L$.

        \item[$b=2-4a$] $ $\\
        In this case,
        \[
        -2=D^{2}=-36a^{2}+36a-8,
        \]
        which has no integer solutions.

        \item[$b=3-4a$] $ $\\
        We obtain
        \[
        -2=D^{2}=-36a^{2}+54a-18,
        \]
        which has no integer solutions.
    \end{itemize}

    \item $D^{2}=0$ and $3\leq D.H\leq 4$.\\
    From $3\leq 4a+b\leq 4$ we obtain $3-4a\leq b\leq 4-4a$.
    Thus:
    \begin{itemize}
        \item[$b=3-4a$] $ $\\
        Then
        \[
        0=D^{2}=-36a^{2}+54a-18,
        \]
        whose only integer solution is $a=1$, $b=-1$, i.e.\ $D=H-L$.
        This divisor has degree $3$ and genus $0$, hence it is a twisted cubic.
        Therefore $X\in\mathcal{F}_2$, which is impossible since $X$ is a very
        general element of $\mathcal{F}_4$. Hence this class does not exist.

        \item[$b=4-4a$] $ $\\
        In this case
        \[
        0=D^{2}=-36a^{2}+72a-32,
        \]
        which has no integer solutions.
    \end{itemize}

    \item $D^{2}=2$ and $D.H=5$.\\
    Then $4a+b=5$, so $b=5-4a$, and
    \[
    2=D^{2}=-36a^{2}+90a-50,
    \]
    which has no integer solutions.

    \item $D^{2}=4$, $D.H=6$, and $|D-H|=\emptyset$.\\
    Then $b=6-4a$ and
    \[
    4=D^{2}=-36a^{2}+108a-72,
    \]
    which has no integer solutions.
\end{itemize}

For item (ii), the surface $X$ is determinantal of type $((1,1),(2,4))$ and its dual
pair is $((1,3),(4,4))$. By Theorem~\ref{teoprinc}, we have two possible resolutions for
$\ii_D$.

\begin{itemize}
    \item[\textbf{Case 1:}]
    \[
    0\to \oo_{\PPP}(-(4+k))\oplus \oo_{\PPP}(-(2+k))
    \to \oo_{\PPP}(-(1+k))^{2}\oplus \oo_{\PPP}(-4)
    \to \ii_D\to0.
    \]

    Using the Betti numbers, the degree and genus of $D$ are computed as
    \begin{equation}\label{ecuagradoygenero}
        d_D=\frac{1}{2}\left(\sum_{i=1}^{2}b_i^2-\sum_{i=0}^{2}a_i^2\right),\qquad
        g_D=1+\frac{1}{6}\left(\sum_{i=1}^{2}b_i^3-\sum_{i=0}^{2}a_i^3\right)-2d_D.
    \end{equation}

    Thus $d_D=4k+1$ and $g_D=2k^2+k$. Assuming $k\geq3$, we have
    \[
    4k+1=D.H=4a+b,\qquad 2k+4k^2-2=D^2=4a^2+2ab-2b^2,
    \]
    which implies $a=k$ and $b=1$. By Proposition~\ref{propclases}, these curves are ACM.

    \item[\textbf{Case 2:}]
    \[
    0\to \oo_{\PPP}(-(4+k))^{2}
    \to \oo_{\PPP}(-(1+k))\oplus \oo_{\PPP}(-(3+k))\oplus \oo_{\PPP}(-4)
    \to \ii_D\to0.
    \]

    Using \eqref{ecuagradoygenero}, we obtain $d_D=4k+3$ and $g_D=2k^2+3k+1$.
    Assuming $k\geq3$, we have
    \[
    4k+3=D.H=4a+b,\qquad 6k+4k^2=D^2=4a^2+2ab-2b^2,
    \]
    which implies $a=k+1$ and $b=-1$. Again, Proposition~\ref{propclases} shows that
    these curves are ACM.
\end{itemize}

Case (iii). Let $D$ be a plane curve in $X$ with class $aH+bL$. Then the degree of $D$ is $4a+b$ and the genus is $\frac{(4a+b-1)(4a+b-2)}{2}$, and this is equal to $\frac{D^{2}-2}{2}$ if and only if $(a=0$ and $b=1)$, $(a=1$ and $b=0)$, or $(a=1$ and $b=-1)$. In the first case, the class of $D$ is $L$, which was already considered in item ii). In the second case, the divisor is $H$, but by hypothesis this cannot occur. Therefore, the class of $D$ is $H-L$, and this class has degree $3$ and genus $1$. 

For the existence of this type of curves in $X$, observe that the complete intersection of a plane containing the line $C_{4}$ with $X$ is the union of $C_{4}$ and a curve $A$. By \cite{mdp2}*{III, Prop. 1.2}, the degree and genus of $A$ are $d_{A}=4(1)-1=3$ and $g_{A}=\frac{(3-1)(1+4-4)}{2}+0=1$, respectively.

For cases (iv) and (v), we treat the remaining cases of item (ii) and the cases given by item (2), (iii) of Theorem \ref{teoprinc}:
\begin{itemize}
    \item[Case 1:] By item (ii), case 1, with $k=0$, we have that $D$ is the line $L$.
    
    \item[Case 2:] By item (ii), case 1, with $k=1$, we have that $D$ is a singular curve of degree $5$ and genus $3$ contained in a singular quadric, with class $H+L$. To obtain the geometric description, let $A$ be a plane curve of degree $3$ contained in $X$ (it exists by item (iii)), and let $Q$ be a quadric surface that contains $A$ and has no common components with $X$ (observe that this surface is the union of two planes). Then the complete intersection of $X$ with $Q$ is the union of $A$ and a curve $D$. By \cite{mdp2}*{III, Prop. 1.2}, the degree and genus of $D$ are $d_{D}=2(4)-3=5$ and $g_{D}=\frac{(5-3)(2+4-4)}{2}+1=3$, respectively.
    
    \item[Case 3:] By case 1 of item (ii), with $k=2$, we have that $D$ is a curve of degree $9$ and genus $10$ contained in a cubic surface, with class $2H+L$. We give a geometric description. Let $A$ be a plane curve of degree $3$ contained in $X$ (it exists by item (iii)), and let $S$ be a cubic surface that contains $A$ and has no common components with $X$. Then the complete intersection of $X$ with $S$ is the union of $A$ and a curve $D$. By \cite{mdp2}*{III, Prop. 1.2}, the degree and genus of $D$ are $d_{D}=3(4)-3=9$ and $g_{D}=\frac{(9-3)(3+4-4)}{2}+1=10$, respectively.

    \item[Case 4:] Applying item (ii), case 2, with $k=0$, we have that $D$ is the plane cubic described in item (iii).
    
    \item[Case 5:] By item (ii), case 2, with $k=1$, $D$ is a curve of degree $7$ and genus $6$ contained in a quadric, with class $2H-L$. A geometric description is as follows. Let $Q$ be a quadric surface that contains the line $C_{4}$ and has no common components with $X$. Then the complete intersection of $X$ with $Q$ is the union of $C_{4}$ and a curve $D$. By \cite{mdp2}*{III, Prop. 1.2}, the degree $d_{D}$ and genus $g_{D}$ of $D$ are $d_{D}=2(4)-1=7$ and $g_{D}=\frac{(7-1)(2+4-4)}{2}+0=6$, respectively.
    
    \item[Case 6:] Using item (ii), case 2, with $k=2$, $D$ is a curve of degree $11$ and genus $15$ contained in a cubic surface, with class $3H-L$. To obtain the geometric description, let $S$ be a cubic surface that contains $C_{4}$ and has no common components with $X$. Then the complete intersection of $X$ with $S$ is the union of $C_{4}$ and a curve $D$. By \cite{mdp2}*{III, Prop. 1.2}, the degree $d_{D}=3(4)-1=11$ and the genus $g_{D}=\frac{(11-1)(3+4-4)}{2}+0=15$.
    
    \item[Case 7:] By Theorem \ref{teoprinc}, item (2) iii), and the pair $((1,1),(2,4))$, we have the following resolutions:
    \begin{itemize}
        \item $$0\to \oo_{\PPP}(-2) \to \oo_{\PPP}(-1)^{2} \to \ii_{D}\to 0,$$
        In this case, $D$ is the line $C_{4}$.
        \item $$0\to \oo_{\PPP}(-6) \to \oo_{\PPP}(-3)^{2} \to \ii_{D}\to 0.$$
        This case coincides with Case 3.
    \end{itemize} 

    \item[Case 8:] By Theorem \ref{teoprinc}, item (2) iii), and the pair $((1,3),(4,4))$, we have the following resolution:
    $$0\to \oo_{\PPP}(-4) \to \oo_{\PPP}(-3)\oplus \oo_{\PPP}(-1) \to \ii_{D}\to 0.$$
    In a similar way to the previous cases, we compute the degree $d_{D}=3$ and the genus $g_{D}=1$, and in this case $D$ is the plane cubic described in item (iii).
\end{itemize}

Finally, to prove item (vi), note that the possible complete intersections contained in a surface $X$ of degree $4$ are either the intersections of $X$ with another surface of any degree, or the complete intersections of two surfaces of degrees $n$ and $m$ with $n,m\in\{1,2,3\}$. 

For the first type, if a curve is the complete intersection of $X$ with a surface of degree $d$, then it has degree $4d$ and genus $2d^{2}+1$. If $D=aH+bL$ is the class of this curve, then
$$
D.H=4a+b=4d \quad \text{and} \quad D^{2}=4a^{2}+2ab-2b^{2}=2d^{2}+1.
$$
This system has a unique solution given by $a=d$ and $b=0$.

Now suppose that the curve is of the second type. Note that the case $m=n=1$ corresponds to a line, the case $m=1$ and $n=3$ (and therefore the case $m=3$ and $n=1$) corresponds to a plane cubic, which was already covered in item (iii). The case $m=n=3$ corresponds to a curve of degree $9$ and genus $10$, which was covered in item (v). 

Observe that the cases $(n,m)=(1,2),(2,1),(2,2)$ are not possible, since the surface $X$ is general and therefore cannot contain conics or elliptic curves of degree $4$. Thus, the only remaining possibility is $m=2$ and $n=3$ (or $m=3$ and $n=2)$). In this case, the complete intersection is a curve of degree $6$ and genus $4$ contained in a quadric surface $Q$. If such a curve existed, we could consider the complete intersection of $Q$ with $X$, and therefore this curve would be linked to a curve of degree $4(2)-6=2$ and genus $\frac{(2-6)(4+2-4)}{2}+4=0$ contained in $X$, which is a contradiction. Hence, such a curve cannot exist.\\

\end{proof}

\begin{rmk}
\begin{enumerate}
    \item Since the equations $D^{2}=-2$ and $D.H=1$ admit a unique integer solution,
    the line $C_4$ is rigid in its numerical class. Therefore, $C_4$ is the unique line
    contained in $X$.

    \item The curve of degree $5$ and genus $3$ appearing in item (iv) cannot be
    contained in a smooth quadric surface, since ACM curves of degree $5$ on a smooth
    quadric have genus $2$. Moreover, this curve does not satisfy Castelnuovo's bound,
    since it is the union of the line $C_4$ with a plane quartic.
\end{enumerate}
\end{rmk}

\subsubsection{\textbf{Quartic surfaces containing a conic}}

Let $X\in\mathcal{F}_{5}$ be a very general smooth element. We denote the class $[C_{5}]$ by $C$. Then $Pic(X)$ is generated by $H$ and $C$, with $H^{2}=4$, $H\cdot C=2$, and $C^{2}=-2$.

\begin{prop}
Let $X=V(F)\in\mathcal{F}_{5}$ be a very general smooth quartic surface in $\PPP$, and let $D$ be a curve on $X$ such that $D\notin |H|$. Let $aH+bC$ be the class of $D$ in $Pic(X)$. Then $D$ is an ACM curve if and only if one of the following cases occurs:
\begin{enumerate}[label=(\roman*)]
    \item $|D-H|=\emptyset$:
    \begin{enumerate}
        \item[(a)] $[D]=C$ or $[D]=H-C$, and the curves in both classes are conics.
    \end{enumerate}

    \item $F$ is a minimal generator of $\ii_{D}$ and $|D-H|\neq\emptyset$:
    \begin{enumerate}
        \item[(a)] $[D]=kH+C$ for $k\in\mathds{Z}$ and $k\geq3$, or $[D]=(k+1)H-C$ for $k\in\mathds{Z}$ and $k\geq3$. The curves in both classes have degree $4k+2$, genus $2k^{2}+2k$, and minimal free resolution
        \[
        0\to \oo_{\PPP}(-(4+k))\oplus \oo_{\PPP}(-(3+k))
        \to \oo_{\PPP}(-(2+k))\oplus \oo_{\PPP}(-(1+k))\oplus \oo_{\PPP}(-4)
        \to \ii_{D}\to 0.
        \]
    \end{enumerate}

    \item $D$ is a curve contained in a quadric surface and $|D-H|\neq\emptyset$:
    \begin{itemize}
        \item[(a)] $[D]=2H$, and $D$ is the complete intersection of $X$ with a quadric; it is a curve of degree $8$ and genus $9$.
        \item[(b)] $[D]=H+C$ or $[D]=2H-C$, and $D$ is a curve of degree $6$ and genus $4$, residual to $C_{5}$ in the complete intersection of $X$ with a quadric surface $Q$ containing $C_{5}$.
    \end{itemize}

    \item $D$ is contained in a cubic surface and $|D-H|\neq\emptyset$:
    \begin{itemize}
        \item[(a)] $[D]=3H$, and $D$ is the complete intersection of $X$ with a cubic; it is a curve of degree $12$ and genus $19$.
        \item[(b)] $[D]=3H-C$ or $[D]=2H+C$, and $D$ is a curve of degree $10$ and genus $12$, residual to the conic $C_{5}$ in the complete intersection of $X$ with a cubic containing $C_{5}$.
    \end{itemize}

    \item $D$ is the complete intersection of $X$ with a surface of degree $d$, with $d\geq4$. The class of this curve is $dH$.
\end{enumerate}
\end{prop}

\begin{proof}
For (i), we follow the conditions of \cite{kwa}*{Thm.1.1}.

\begin{itemize}
    \item $D^{2}=-2$ and $1\leq D\cdot H\leq3$.\\
    From the second condition we obtain $1\leq4a+2b\leq3$, hence
    $1-4a\leq2b\leq3-4a$. Since $b\in\mathbb{Z}$, we have $b=1-2a$.
    From the first equation,
    \[
    -2=D^{2}=4a^{2}+4ab-2b^{2}
    =4a^{2}+4a(1-2a)-2(1-2a)^{2},
    \]
    which gives $6a(1-a)=0$. Therefore, the only integer solutions are
    $(a,b)=(0,1)$, corresponding to $D=C$, and $(a,b)=(1,-1)$, corresponding to
    $D=H-C$.

    \item $D^{2}=0$ and $3\leq D\cdot H\leq4$.\\
    From $3\leq4a+2b\leq4$ we obtain $b=2-2a$. Substituting, we get
    \[
    0=D^{2}=-12a^{2}+24a-4,
    \]
    which has no integer solutions.

    \item $D^{2}=2$ and $D\cdot H=5$.\\
    Then $4a+2b=5$, which has no integer solutions.

    \item $D^{2}=4$, $D\cdot H=6$, and $|D-H|=|2H-D|=\emptyset$.\\
    From $4a+2b=6$ we have $b=3-2a$. Substituting,
    \[
    4=D^{2}=-12a^{2}+36a-18,
    \]
    which has no integer solutions.
\end{itemize}

For (ii), the surface $X$ is a determinantal quartic surface of type $((1,2),(3,4))$ and is autodual. By Theorem~\ref{teoprinc}, the resolution of $D$ is
\[
0\to \oo_{\PPP}(-(4+k))\oplus \oo_{\PPP}(-(3+k))
\to \oo_{\PPP}(-(2+k))\oplus \oo_{\PPP}(-(1+k))\oplus \oo_{\PPP}(-4)
\to \ii_{D}\to 0.
\]
Using the formulas \eqref{ecuagradoygenero}, we obtain $d_{D}=4k+2$ and $g_{D}=2k^{2}+2k$.
Since the cases where $D$ is contained in a surface of smaller degree are treated separately, we assume $k\geq3$.
Moreover,
\[
4k+2=D\cdot H=4a+2b,\qquad
4k+4k^{2}-2=D^{2}=4a^{2}+4ab-2b^{2},
\]
which implies $a=k$, $b=1$, and $D=kH+C$, or $a=k+1$, $b=-1$, and $D=(k+1)H-C$.

We now verify that the curves in both classes are ACM. Suppose $|D-H|\neq\emptyset$ and $D=(k+4)H+C$ with $k\geq-2$. Then
$nH-D=(n-(k+4))H-C$, and since $C$ is ACM,
\[
H^{1}(\oo_{X}(nH-D))=H^{1}(\oo_{X}((n-(k+4))H-C))=0
\]
for all $n$, hence $D$ is ACM.

Similarly, if $|D-H|\neq\emptyset$ and $D=(k+5)H-C$ with $k\geq-2$, then
$nH-D=(n-(k+5))H+C$, and since $-C$ is ACM,
\[
H^{1}(\oo_{X}(nH-D))=H^{1}(\oo_{X}((n-(k+5))H+C))=0
\]
for all $n$, and therefore $D$ is ACM.

We now verify that there are no plane curves different from conics or plane quartic curves (hyperplane sections). Let $D$ be a plane curve on $X$ with class $aH+bC$. Its degree is $4a+2b$, and its genus is
$\frac{(4a+2b-1)(4a+2b-2)}{2}$.
This equals $\frac{D^{2}-2}{2}$ if and only if $(a,b)=(0,1)$, $(1,0)$, or $(1,-1)$.
Thus $D$ has class $C$, $H$, or $H-C$, and the claim follows.

For items (iii) and (iv), we consider the remaining cases of (ii) and the cases arising from items (2)(ii) and (2)(iii) of Theorem~\ref{teoprinc}:
\begin{itemize}
    \item[\textbf{Case 1:}] By (ii) with $k=0$, $D$ is a conic.

    \item[\textbf{Case 2:}] By (ii) with $k=1$, $D$ is a curve of degree $6$ and genus $4$ contained in a quadric, with class $H+C$ or $2H-C$. Geometrically, let $Q$ be a quadric surface containing $C_{5}$ and having no common components with $X$. Then $X\cap Q=C_{5}\cup D$, and by \cite{mdp2}*{III, Prop.1.2}, we have $d_{D}=6$ and $g_{D}=4$.

    \item[\textbf{Case 3:}] By (ii) with $k=2$, $D$ is a curve of degree $10$ and genus $12$ contained in a cubic surface, with class $2H+C$ or $3H-C$. Let $S$ be a cubic surface containing $C_{5}$ and having no common components with $X$. Then $X\cap S=C_{5}\cup D$, and by \cite{mdp2}*{III, Prop.1.2}, we have $d_{D}=10$ and $g_{D}=12$.

    \item[\textbf{Case 4:}] By Theorem~\ref{teoprinc}(2)(iii) and the pair $((1,2),(3,4))$, we obtain the following resolutions:
    \begin{itemize}
        \item
        \[
        0\to \oo_{\PPP}(-3)\oplus \oo_{\PPP}
        \to \oo_{\PPP}(-1)\oplus \oo_{\PPP}(-2)
        \to \ii_{D}\to 0,
        \]
        which corresponds to a conic.
        \item
        \[
        0\to \oo_{\PPP}(-5)\oplus \oo_{\PPP}
        \to \oo_{\PPP}(-2)\oplus \oo_{\PPP}(-3)
        \to \ii_{D}\to 0,
        \]
        which coincides with Case~2.
    \end{itemize}

    In all cases, the curves are ACM. Indeed, if $D=mH+C$, then
    $nH-D=(n-m)H-C$, and since $C$ is ACM,
    $H^{1}(\oo_{X}(nH-D))=0$ for all $n$.
    Similarly, if $D=mH-C$, then
    $nH-D=(n-m)H+C$, and since $-C$ is ACM,
    $H^{1}(\oo_{X}(nH-D))=0$ for all $n$.
\end{itemize}

Finally, we prove (v). As in the previous proposition, the only possibility is that the curve is the complete intersection of $X$ with another surface of degree $d$, in which case $a=d$ and $b=0$. For the remaining options, if the curve is the complete intersection of surfaces of degrees $m$ and $n$, we have:
for $(m,n)=(1,2)$ or $(2,1)$, the curve is a conic with class $C$;
for $(m,n)=(1,1)$, $(1,3)$, $(3,1)$, or $(3,3)$, these cases are not possible, since they would give a line or a plane cubic, whose residual curve under the complete intersection would be a line contained in $X$, which is impossible.
The remaining cases $(m,n)=(2,3)$ or $(3,2)$ give a curve of degree $6$ and genus $4$, which is covered in item (iii).\\
\end{proof}

\subsubsection{\textbf{Quartic surfaces containing an ACM curve of degree $4$ and genus $1$}}

Let $X\in\mathcal{F}_{3}$ be a very general smooth element. We denote the class $[C_{3}]$ by $C$. Then $Pic(X)$ is generated by $H$ and $C$, with $H^{2}=4$, $H\cdot C=4$, and $C^{2}=0$.

\begin{prop}
Let $X=V(F)\in\mathcal{F}_{3}$ be a very general smooth quartic surface in $\PPP$, and let $D$ be a curve on $X$ such that $D\notin |H|$. Let $aH+bC$ be the class of $D$ in $Pic(X)$. Then $D$ is an ACM curve if and only if one of the following cases occurs:
\begin{enumerate}[label=(\roman*)]
    \item $|D-H|=\emptyset$:
    \begin{enumerate}
        \item[(a)] $[D]=C$ or $[D]=2H-C$, and in both cases the curves in these classes are ACM curves of degree $4$ and genus $1$.
    \end{enumerate}

    \item $F$ is a minimal generator of $\ii_{D}$ and $|D-H|\neq\emptyset$:
    \begin{enumerate}
        \item[(a)] $[D]=(k+1)H-C$ for $k\in\mathds{Z}$ with $k\geq3$, or $[D]=(k-1)H+C$ for $k\in\mathds{Z}$ with $k\geq3$.  
        In both cases, the curves have degree $4k$, genus $2k^{2}-1$, and minimal free resolution
        \[
        0\to \oo_{\PPP}(-(3+k))^{2}
        \to \oo_{\PPP}(-(1+k))^{2}\oplus \oo_{\PPP}(-4)
        \to \ii_{D}\to 0.
        \]
    \end{enumerate}

    \item $D$ is a curve contained in a quadric surface and $|D-H|\neq\emptyset$:
    \begin{itemize}
        \item[(a)] $[D]=2H$, and $D$ is the complete intersection of $X$ with a quadric surface; it is a curve of degree $8$ and genus $9$.
    \end{itemize}

    \item $D$ is a curve contained in a cubic surface and $|D-H|\neq\emptyset$:
    \begin{itemize}
        \item[(a)] $[D]=3H$, and $D$ is the complete intersection of $X$ with a cubic surface; in this case $D$ is a curve of degree $12$ and genus $19$.
        \item[(b)] $[D]=H+C$ or $[D]=3H-C$, and in both cases $D$ is a curve of degree $8$ and genus $7$, residual to the curve $C_{3}$ in the complete intersection of $X$ with a cubic surface $S$ containing $C_{3}$.
    \end{itemize}

    \item $D$ is the complete intersection of $X$ with a surface of degree $d$, with $d\geq4$. The class of this curve is $dH$.
\end{enumerate}
\end{prop}

\begin{proof}
For (i), as before, we follow the conditions of \cite{kwa}*{Thm.1.1}.

\begin{itemize}
    \item $D^{2}=-2$ and $1\leq D\cdot H\leq3$.\\
    From the second condition we obtain $1\leq4a+4b\leq3$, hence
    $1-4a\leq4b\leq3-4a$. Since $b\in\mathbb{Z}$, this inequality has no solution.

    \item $D^{2}=0$ and $3\leq D\cdot H\leq4$.\\
    From $3\leq4a+4b\leq4$ we obtain $b=1-a$.
    Substituting into the first equation,
    \[
    0=D^{2}=4a^{2}+8ab=4a^{2}+8a(1-a)=4a(2-a),
    \]
    which has two solutions: $a=0$, $b=1$, corresponding to $D=C$, and
    $a=2$, $b=-1$, corresponding to $D=2H-C$.

    \item $D^{2}=2$ and $D\cdot H=5$.\\
    Then $D\cdot H=4a+4b=2(2a+2b)=5$, which has no integer solution.

    \item $D^{2}=4$, $D\cdot H=6$, and $|D-H|=|2H-D|=\emptyset$.\\
    From $4a+4b=6$ we obtain $2(a+b)=3$, which has no integer solution.
\end{itemize}

For (ii), the surface $X$ is a determinantal quartic surface of type $((1,1),(3,3))$ and is autodual. Therefore, by Theorem~\ref{teoprinc}, the resolution of $D$ is
\[
0\to \oo_{\PPP}(-(3+k))^{2}
\to \oo_{\PPP}(-(1+k))^{2}\oplus \oo_{\PPP}(-4)
\to \ii_{D}\to 0.
\]
Using the formulas \eqref{ecuagradoygenero}, we compute
$d_{D}=4k$ and $g_{D}=2k^{2}-1$.
Since the cases where $D$ is contained in a surface of lower degree are treated separately, we assume $k\geq3$.
Moreover,
\[
4k=D\cdot H=4a+4b,\qquad
4k^{2}-4=D^{2}=4a^{2}+8ab,
\]
from which it follows that $a=k+1$, $b=-1$, and $D=(k+1)H-C$, or
$a=k-1$, $b=1$, and $D=(k-1)H+C$.

We now verify that the curves in both classes are ACM. Suppose $|D-H|\neq\emptyset$ and $D=(k+1)H-C$.
Then $nH-D=(n-(k+1))H-C$, and since $C$ is ACM,
\[
H^{1}(\oo_{X}(nH-D))=H^{1}(\oo_{X}((n-(k+1))H-C))=0
\]
for all $n$, hence $D$ is ACM.

Now let $D=(k-1)H+C$ with $|D-H|\neq\emptyset$.
Then $nH-D=(n-(k+1))H+C$, and since $-C$ is ACM, the same argument shows that $D$ is ACM.

We now verify that there are no plane curves other than plane quartic curves (hyperplane sections).
Let $D$ be a plane curve on $X$ with class $aH+bC$.
Then the degree of $D$ is $4a+4b$, and its genus is
\[
\frac{(4a+4b-1)(4a+4b-2)}{2}.
\]
This equals $\frac{D^{2}-2}{2}$ if and only if $a=1$ and $b=0$, hence $D\in|H|$, and the claim follows.

For (iii) and (iv), we consider the remaining cases of (ii) and the cases given by items (2)(ii) and (2)(iii) of Theorem~\ref{teoprinc}:
\begin{itemize}
    \item[\textbf{Case 1:}] Item (ii) with $k=0$. This is not possible since the degree is $0$.

    \item[\textbf{Case 2:}] Item (ii) with $k=1$. In this case, the curve is $C_{3}$.

    \item[\textbf{Case 3:}] Item (ii) with $k=2$. Then $D$ is a curve of degree $8$ and genus $7$ contained in a cubic surface, with class $H+C$ or $3H-C$.  
    Geometrically, let $S$ be a cubic surface containing $C_{3}$ and having no common components with $X$. Then
    $X\cap S=C_{3}\cup D$, and by \cite{mdp2}*{III,Prop.1.2}, we obtain
    $d_{D}=8$ and $g_{D}=7$.

    \item[\textbf{Case 4:}] By Theorem~\ref{teoprinc}(2)(iii) and the pair $((1,1),(3,3))$, we obtain the resolution
    \[
    0\to \oo_{\PPP}(-4)\oplus \oo_{\PPP}
    \to \oo_{\PPP}(-2)\oplus \oo_{\PPP}(-2)
    \to \ii_{D}\to 0.
    \]
    As before, in all cases the curves in both classes are ACM. Indeed, if $D=mH+C$, then
    $nH-D=(n-m)H-C$, and since $C$ is ACM,
    $H^{1}(\oo_{X}(nH-D))=0$ for all $n$.
    Similarly, if $D=mH-C$, then
    $nH-D=(n-m)H+C$, and since $-C$ is ACM,
    $H^{1}(\oo_{X}(nH-D))=0$ for all $n$.
\end{itemize}

Finally, we prove (v). As in Proposition~\ref{prop1}, if the curve is the complete intersection of $X$ with a surface of degree $d$, then its class is $dH$.
Now suppose that the curve is the complete intersection of surfaces of degrees $m$ and $n$, with $m\leq n$.
If $m=n=2$, then the curve is the complete intersection of two quadric surfaces, hence it is $C_{3}$.
The case $m=n=1$ gives a line; $(m,n)=(1,2)$ gives a conic; $(m,n)=(2,3)$ gives a curve residual to a conic;
$(m,n)=(1,3)$ gives a curve residual to a line; and $m=n=3$ gives a curve residual to a plane cubic.
All these cases are not possible.\\
\end{proof}

\subsubsection{\textbf{Quartic surfaces containing a twisted cubic}}
Let $X \in \mathcal{F}_{2}$ be a very general smooth element. We denote by $C$ the class $[C_{2}]$. Then $Pic(X)$ is generated by $H$ and $C$, with intersection numbers
\[
H^{2} = 4, \quad H \cdot C = 3, \quad C^{2} = -2.
\]

\begin{prop}
Let $X = V(F) \in \mathcal{F}_{2}$ be a very general smooth quartic surface in $\PPP$, and let $D \subset X$ be a curve such that $D \notin |H|$. Let $aH + bC$ be the class of $D$ in $Pic(X)$. Then $D$ is an ACM curve if and only if one of the following cases occurs:
\begin{enumerate}[label=(\roman*)]
    \item $D$ is a complete intersection.

    \item $|D-H| = \emptyset$:
    \begin{enumerate}
        \item[(a)] $[D] = C$, and the curves in this class are twisted cubics.
        \item[(b)] $[D] = 2H - C$, and $D$ has degree $5$ and genus $2$, with minimal free resolution
        \[
        0 \to \mathcal{O}_{\PPP}(-4)^{2} \to \mathcal{O}_{\PPP}(-2) \oplus \mathcal{O}_{\PPP}(-3)^{2} \to \mathscr{I}_{D} \to 0.
        \]
    \end{enumerate}

    \item $F$ is a minimal generator of $\mathscr{I}_{D}$ and $|D-H| \neq \emptyset$:
    \begin{enumerate}
        \item[(a)] $[D] = (k-1)H + C$, for $k \in \mathds{Z}$ and $k \geq 3$. Curves in this class have degree $4k-1$, genus $2k^{2}-k-1$, and minimal free resolution
        \[
        0 \to \oo_{\PPP}(-(k+2))^{2} \oplus \oo_{\PPP}(-(k+3)) \to \oo_{\PPP}(-(k+1))^{3} \oplus \oo_{\PPP}(-4) \to \ii_{D} \to 0.
        \]

        \item[(b)] $[D] = (k+1)H - C$, for $k \in \mathds{Z}$ and $k \geq 3$. Curves in this class have degree $4k+1$, genus $2k^{2}+k-1$, and minimal free resolution
        \[
        0 \to \oo_{\PPP}(-(k+3))^{3} \to \oo_{\PPP}(-(k+1)) \oplus \oo_{\PPP}(-(k+2))^{2} \oplus \oo_{\PPP}(-4) \to \ii_{D} \to 0.
        \]
    \end{enumerate}

    \item $D$ is contained in a quadric surface and $|D-H| \neq \emptyset$:
    \begin{itemize}
        \item[(a)] $[D] = 2H$, and $D$ is the complete intersection of $X$ with a quadric surface, hence a curve of degree $8$ and genus $9$.
        \item[(b)] $[D] = 2H - C$, and $D$ is a curve of degree $5$ and genus $2$, residual to the twisted cubic $C_{2}$ in the complete intersection of $X$ with a quadric surface $Q$ containing $C_{2}$.
    \end{itemize}

    \item $D$ is contained in a cubic surface and $|D-H| \neq \emptyset$:
    \begin{itemize}
        \item[(a)] $[D] = 3H$, and $D$ is the complete intersection of $X$ with a cubic surface, hence a curve of degree $12$ and genus $19$.
        \item[(b)] $[D] = H + C$, and $D$ is a curve of degree $7$ and genus $5$, residual to an ACM curve $A \subset X$ of degree $5$ and genus $2$ (see item (iv)) in the complete intersection of $X$ with a cubic surface $S$ containing $A$.
        \item[(c)] $[D] = 3H - C$, and $D$ is a curve of degree $9$ and genus $9$, residual to the twisted cubic $C_{2}$ in the complete intersection of $X$ with a cubic surface $S$ containing $C_{2}$.
    \end{itemize}

    \item $D$ is the complete intersection of $X$ with a surface of degree $d \geq 4$. In this case, the class of $D$ is $dH$.
\end{enumerate}
\end{prop}
\begin{proof}
For (i), as before, we follow the conditions of Theorem 1.1 of \cite{kwa}.
\begin{itemize}
        \item $D^{2}=-2$ and $1\leq D.H\leq 3$.\\
    From the second condition, we have $1\leq D.H=4a+3b\leq 3$, hence
    $\frac{1}{3}-\frac{4a}{3}\leq b\leq 1-\frac{4a}{3}$. We consider three cases:
    \begin{itemize}
        \item[$a=3m$]  
        $$\frac{1}{3}-\frac{4a}{3}=\frac{1-12m}{3}\leq b \leq 1-\frac{4a}{3}=1-4m.$$
        Thus, the only possible choice is $a=3m$ and $b=1-4m$. From the first equation we obtain
        $-2=D^{2}=4a^{2}+6ab-2b^{2}$, hence
        $$-1=2a^{2}+3ab-b^{2}=2(3m)^{2}+3(3m)(1-4m)-(1-4m)^{2}=18m^{2}+9m-36m^{2}-1+8m-16m^{2}$$
        and therefore
        $$m(17-34m)=0.$$
        Since the equation $17-34m=0$ has no integer solutions, the only solution is $m=0$, hence $a=0$, $b=1$, and $D=C$.

        \item[$a=3m+1$] 
        $$\frac{1}{3}-\frac{4a}{3}=\frac{-12m-3}{3}=-4m-1\leq b \leq 1-\frac{4a}{3}=\frac{12m-1}{3}.$$
        Thus, the only possible choice is $a=3m+1$ and $b=-1-4m$. From the first equation we obtain
\begin{align*}
    -1&=2a^{2}+3ab-b^{2}=2(3m+1)^{2}+3(3m+1)(-1-4m)-(-1-4m)^{2}\\
    &=18m^{2}+12m+2-36m^{2}-9m-12m-3-16m^{2}-12m-1,
\end{align*}
and therefore
$$m(-21-34m)=1,$$
which has no integer solutions.

        \item[$a=3m+2$] 
        $$\frac{1}{3}-\frac{4a}{3}=\frac{-7-12m}{3}\leq b \leq 1-\frac{4a}{3}=\frac{-12m-5}{3}.$$
        Hence, the only possible choice is $b=\frac{-12m-6}{3}=-4m-2$. From the first equation we obtain
\begin{align*}
    -1&=2a^{2}+3ab-b^{2}=2(3m+2)^{2}+3(3m+2)(-2-4m)-(-2-4m)^{2}\\
    &=18m^{2}+24m+8-36m^{2}-18m-24m-12-16m^{2}-16m-4,
\end{align*}
and therefore
$$34m(-1-m)=7,$$
which has no integer solutions.
\end{itemize}

        \item $D^{2}=0$ and $3\leq D.H \leq 4$. \\
        From the second condition, we have $3\leq D.H=4a+3b\leq 4$, hence
        $1-\frac{4a}{3}\leq b\leq \frac{4}{3}-\frac{4a}{3}$. Again, we consider three cases:
\begin{itemize}
        \item[$a=3m$]  
        $$1-\frac{4a}{3}=1-4m\leq b \leq\frac{4}{3}-\frac{4a}{3}=\frac{-12m-4}{3}.$$
        Thus, the only possible choice is $a=3m$ and $b=1-4m$. From the first equation we obtain
        $0=D^{2}=4a^{2}+6ab-2b^{2}$, hence
        $$0=2a^{2}+3ab-b^{2}=2(3m)^{2}+3(3m)(1-4m)-(1-4m)^{2}=18m^{2}+9m-36m^{2}-1+8m-16m^{2}$$
        and therefore
        $$m(17-34m)=1,$$
        which has no integer solutions.

        \item[$a=3m+1$] 
        $$1-\frac{4a}{3}=\frac{-12m-1}{3}\leq b \leq \frac{4}{3}-\frac{4a}{3}=\frac{-12m}{3}=-4m.$$
        Hence, the only possible choice is $a=3m+1$ and $b=-4m$. From the first equation we obtain
\begin{align*}
    0&=2a^{2}+3ab-b^{2}=2(3m+1)^{2}+3(3m+1)(-4m)-(-4m)^{2}\\
    &=18m^{2}+12m+2-36m^{2}-12m-16m^{2},
\end{align*}
and therefore
$$-34m^{2}=-2,$$
which has no integer solutions.

        \item[$a=3m+2$] 
        $$1-\frac{4a}{3}=\frac{-12m-5}{3}\leq b \leq \frac{4}{3}-\frac{4a}{3}=\frac{-12m-4}{3}.$$
        This inequality has no integer solutions.
\end{itemize}

        \item $D^{2}=2$ and $D.H=5$.\\
        From the second equation, we obtain $a=3m+2$ and $b=-1-4m$. Then
\begin{align*}
    1&=2a^{2}+3ab-b^{2}=2(3m+2)^{2}+3(3m+2)(-1-4m)-(-1-4m)^{2}\\
    &=18m^{2}+12m+8-9m-36m^{2}-6-24m-1-8m-16m^{2}\\
    &=-34m^{2}-21m+1.
\end{align*}
        Hence, the only solution is $m=0$, and therefore $a=2$, $b=-1$, and $D=2H-C$, which is a curve of degree $5$ and genus $2$.

        \item $D^{2}=4$, $D.H=6$ and $|D-H|=|2H-D|=\emptyset$.\\
        From the second equation, we have $D.H=4a+3b=6$, which has no integer solutions.
\end{itemize}

For (ii), we have that $X$ is determinantal of type $((1,1,1),(2,2,3))$ and its dual pair is $((1,2,2),(3,3,3))$. Therefore, by Theorem \ref{teoprinc}, we have two possible resolutions for $D$:
\begin{itemize}
    \item[Case 1:]  
    $$0\to \oo_{\PPP}(-(3+k))\oplus \oo_{\PPP}(-(2+k))^{2}\to \oo_{\PPP}(-(1+k))^{3}\oplus \oo_{\PPP}(-4)\to \ii_{D}\to 0.$$
    As before, using formulas \eqref{ecuagradoygenero}, we compute the degree $d_{D}=4k-1$ and the genus $g_{D}=2k^{2}-k-1$.\\
    Since we will treat the cases in which $D$ is contained in a surface of lower degree in another item, we assume $k\geq 3$. On the other hand,
$$4k-1=d_{D}=D.H=4a+3b \quad \text{and} \quad -2k+4k^{2}-2=2g_{D}-2=D^{2}=4a^{2}+6ab-2b^{2}.$$
    It follows that $a=k-1$ and $b=1$ is the only solution. By Proposition \ref{propclases}, curves with this class are ACM curves.

 \item[Case 2:]  
 $$0\to \oo_{\PPP}(-(3+k))^{3}\to \oo_{\PPP}(-(1+k))\oplus \oo_{\PPP}(-(2+k))^{2}\oplus \oo_{\PPP}(-4)\to \ii_{D}\to 0.$$
    Using the Betti numbers and formulas \eqref{ecuagradoygenero}, we compute the degree $d_{D}=4k+1$ and the genus $g_{D}=2k^{2}+k-1$.\\
    Again, we assume $k\geq 3$. On the other hand,
$$4k+1=d_{D}=D.H=4a+3b \quad \text{and} \quad 2k+4k^{2}-4=2g_{D}-2=D^{2}=4a^{2}+6ab-2b^{2}.$$
    It follows that $a=k+1$ and $b=-1$ is the only solution. By Proposition \ref{propclases}, curves with this class are ACM curves.
\end{itemize}

We must verify that there are no plane curves other than plane quartic curves (hyperplane sections). Let $D$ be a plane curve in $X$ with class $aH+bL$. Then the degree of $D$ is $4a+3b$ and the genus is $\frac{(4a+3b-1)(4a+3b-2)}{2}$, and this is equal to $\frac{D^{2}-2}{2}$ if and only if $a=1$ and $b=0$. Therefore, $D\in |H|$, and these are all the possibilities.

For items (iii) and (iv), we treat the remaining cases of (iii) and the cases given by item (2), (iii) of Theorem \ref{teoprinc}:
\begin{itemize}
    \item[Case 1:] By item (ii), case 1 with $k=0$, the intersection consists of three planes, and therefore it is not a curve.
    
    \item[Case 2:] By item (ii), case 1 with $k=1$, $D$ is a singular curve of degree $3$ and genus $0$, namely the twisted cubic $C_{2}$.
    
    \item[Case 3:] By item (ii), case 1 with $k=2$, $D$ is a curve of degree $7$ and genus $5$ contained in a cubic surface, with class $H+C$. A geometric description is as follows. Let $A$ be a curve of degree $5$ and genus $2$ contained in $X$ (we will prove its existence in Case 5), and let $S$ be a cubic surface that contains $A$ and has no common components with $X$. Then the complete intersection of $X$ with $S$ is the union of $A$ and a curve $D$. By \cite{mdp2}*{III, Prop. 1.2}, the degree and genus of $D$ are $d_{D}=3(4)-5=7$ and $g_{D}=\frac{(7-5)(3+4-4)}{2}+2=5$, respectively.
    
    \item[Case 4:] By item (ii), case 2, with $k=0$, this case is not possible.
    
    \item[Case 5:] By item (ii), case 2, with $k=1$, $D$ is a curve of degree $5$ and genus $2$ contained in a quadric surface, with class $2H-C$. A geometric description is as follows. Let $Q$ be a quadric surface that contains the twisted cubic $C_{2}$ and has no common components with $X$. Then the complete intersection of $X$ with $Q$ is the union of $C_{2}$ and a curve $D$. By \cite{mdp2}*{III, Prop. 1.2}, the degree and genus of $D$ are $d_{D}=2(4)-3=5$ and $g_{D}=\frac{(5-3)(2+4-4)}{2}+0=2$, respectively.
    
    \item[Case 6:] By item (ii), case 2, with $k=2$, $D$ is a curve of degree $9$ and genus $9$ contained in a cubic surface, with class $3H-C$. A geometric description is as follows. Let $S$ be a cubic surface that contains $C_{2}$ and has no common components with $X$. Then the complete intersection of $X$ with $S$ is the union of $C_{2}$ and a curve $D$. By \cite{mdp2}*{III, Prop. 1.2}, the degree and genus of $D$ are $d_{D}=3(4)-3=9$ and $g_{D}=\frac{(9-3)(3+4-4)}{2}+0=9$, respectively.

    \item[Case 7:] By Theorem \ref{teoprinc}, item (2) iii), and the pair $((1,1,1),(2,2,3))$, we have the following resolutions:
    \begin{itemize}
        \item $$0\to \oo_{\PPP}(-3)^{2}\to \oo_{\PPP}(-2)^{3} \to \ii_{D}\to 0.$$
        In this case, $D$ is the twisted cubic $C_{2}$.
        \item $$0\to \oo_{\PPP}(-5)\oplus \oo_{\PPP}(-4)\to \oo_{\PPP}(-3)^{3} \to \ii_{D}\to 0.$$
        This case coincides with Case 3.
    \end{itemize} 

    \item[Case 8:] By Theorem \ref{teoprinc}, item (2) iii), and the pair $((1,2,2),(3,3,3))$, we have the following resolution:
    $$0\to \oo_{\PPP}(-4)^{2}\to \oo_{\PPP}(-3)^{2}\oplus \oo_{\PPP}(-2) \to \ii_{D}\to 0,$$
    and this case coincides with Case 5.
\end{itemize}

For the last item, as we saw in the proof of the previous propositions, curves that are complete intersections of surfaces of degrees $m$ and $n$ with $m,n\leq 3$ do not occur. Therefore, the only possible curves are complete intersections of $X$ with another surface of degree $d$, in which case the class is $dH$.\\
\end{proof}

\subsubsection{\textbf{Quartic surfaces containing an ACM curve of degree $6$ and genus $3$}}
Let $X \in \mathcal{F}_{1}$ be a very general smooth element. We denote by $C$ the class $[C_{1}]$. Then $Pic(X)$ is generated by $H$ and $C$, with intersection numbers
\[
H^{2} = 4, \quad H \cdot C = 6, \quad C^{2} = 4.
\]

\begin{prop}
Let $X = V(F) \in \mathcal{F}_{1}$ be a very general smooth quartic surface in $\PPP$, and let $D \subset X$ be a curve such that $D \notin |H|$. Let $aH + bC$ be the class of $D$ in $Pic(X)$. Then $D$ is an ACM curve if and only if one of the following cases occurs:
\begin{enumerate}[label=(\roman*)]
    \item $|D-H| = \emptyset$:
    \begin{enumerate}
        \item[(a)] $[D] = C$ or $[D] = 3H - C$, and in both cases the curves in these classes are equivalent to the curve $C_{1}$.
    \end{enumerate}

    \item $F$ is a minimal generator of $\ii_{D}$ and $|D-H| \neq \emptyset$:
    \begin{enumerate}
        \item $[D] = (k-2)H + C$ for $k \in \mathds{Z}$ with $k \geq 3$, or $[D] = (k+1)H - C$ for $k \in \mathds{Z}$ with $k \geq -1$. In both cases, the curves have degree $4k-2$, genus $2k^{2}-2k-1$, and minimal free resolution
        \[
        0 \to \oo_{\PPP}(-(k+2))^{4} \to \oo_{\PPP}(-(k+1))^{4} \oplus \oo_{\PPP}(-4) \to \ii_{D} \to 0.
        \]
    \end{enumerate}

    \item $D$ is a curve contained in a quadric surface and $|D-H| \neq \emptyset$:
    \begin{itemize}
        \item[(a)] $[D] = 2H$, and $D$ is the complete intersection of $X$ with a quadric surface, hence a curve of degree $8$ and genus $9$.
    \end{itemize}

    \item $D$ is a curve contained in a cubic surface and $|D-H| \neq \emptyset$:
    \begin{itemize}
        \item[(a)] $[D] = 3H$, and $D$ is the complete intersection of $X$ with a cubic surface, hence a curve of degree $12$ and genus $19$.
    \end{itemize}

    \item $D$ is the complete intersection of $X$ with a surface of degree $d \geq 4$. In this case, the class of $D$ is $dH$.
\end{enumerate}
\end{prop}

\begin{proof}
For (i), as before, we follow the conditions of Theorem~1.1 in \cite{kwa}.
\begin{itemize}
    \item $D^{2} = -2$ and $1 \leq D \cdot H \leq 3$.\\
    From the second condition we have $1 \leq 4a + 6b \leq 3$, hence $1 - 6b \leq 4a \leq 3 - 6b$. Since $a$ is an integer, there are two possible cases:
    \begin{itemize}
        \item[*] $b = 4m + 1$ and $a = -1 - 6m$.\\
        From the first equation we obtain
        \begin{align*}
            -2 &= D^{2} = 4a^{2} + 12ab + 4b^{2} \\
               &= 4(-1-6m)^{2} + 12(-1-6m)(4m+1) + 4(4m+1)^{2} \\
               &= 2(-2 - 20m - 76m^{2}),
        \end{align*}
        which has no integer solutions.

        \item[*] $b = 4m + 3$ and $a = -4 - 6m$.\\
        From the first equation we obtain
        \begin{align*}
            -2 &= D^{2} = 4(-4-6m)^{2} + 12(-4-6m)(4m+3) + 4(4m+3)^{2} \\
               &= 2(-22 - 60m - 40m^{2}),
        \end{align*}
        which has no integer solutions.
    \end{itemize}

    \item $D^{2} = 0$ and $3 \leq D \cdot H \leq 4$.\\
    From the second condition we have $3 \leq 4a + 6b \leq 4$, hence $3 - 6b \leq 4a \leq 4 - 6b$. Since $a$ is an integer, we obtain $b = 4m + 2$ and $a = -2 - 6m$. From the first equation we have
    \begin{align*}
        0 &= D^{2} = 4(-2-6m)^{2} + 12(-2-6m)(4m+2) + 4(4m+2)^{2} \\
          &= 2(-8 - 40m - 40m^{2}),
    \end{align*}
    which has no integer solutions.

    \item $D^{2} = 2$ and $D \cdot H = 5$.\\
    From the second condition we have $4a + 6b = 5$. Since $a$ is an integer, we obtain $a = 3m + 1$ and $b = -2m$. From the first equation we have
    \begin{align*}
        2 &= D^{2} = 4(3m+1)^{2} + 12(3m+1)(-2m) + 4(-2m)^{2} \\
          &= 2(2 - 10m^{2}),
    \end{align*}
    which has no integer solutions.

    \item $D^{2} = 4$, $D \cdot H = 6$, and $|D-H| = |2H-D| = \emptyset$.\\
    From the second equation we have $4a + 6b = 6$, hence $2a = 3(1-b)$, which implies $a = 3m$ and $b = 1 - 2m$. From the first equation we obtain
    \begin{align*}
        4 &= D^{2} = 4(3m)^{2} + 12(3m)(1-2m) + 4(1-2m)^{2} \\
          &= 2(2 + 10m - 10m^{2}),
    \end{align*}
    whose integer solutions are $m=0$ and $m=1$. Hence $D=C$ or $D=3H-C$.
\end{itemize}

For (ii), $X$ is a determinantal quartic surface of type $((1,1,1,1),(2,2,2,2))$ and is autodual. Therefore, by Theorem~\ref{teoprinc}, the minimal free resolution of $D$ is
\[
0 \to \oo_{\PPP}(-(k+2))^{4} \to \oo_{\PPP}(-(k+1))^{4} \oplus \oo_{\PPP}(-4) \to \ii_{D} \to 0.
\]
Using formulas~\eqref{ecuagradoygenero}, we obtain $d_{D} = 4k - 2$ and $g_{D} = 2k^{2} - 2k - 1$.\\
Since we will deal with cases where $D$ is contained in a surface of lower degree, we assume $k \geq 3$. Moreover,
\[
4k-2 = D \cdot H = 4a + 6b, \qquad 4k^{2} - 4k - 4 = D^{2} = 4a^{2} + 12ab + 4b^{2},
\]
which implies either $D = (k-2)H + C$ or $D = (k+1)H - C$.

We now verify that the curves in both classes are ACM. Suppose $|D-H| \neq \emptyset$ and $D = (k+1)H - C$. Then
\[
nH - D = (n - (k+1))H - C,
\]
and since $C$ is ACM, we have $H^{1}(\oo_{X}(nH-D)) = 0$ for all $n$, hence $D$ is ACM. The same argument applies to $D = (k-2)H + C$.

Finally, we verify that there are no plane curves other than hyperplane sections. Let $D$ be a plane curve in $X$ with class $aH + bL$. Then its degree is $4a + 6b$, and its genus is $\frac{(4a+6b-1)(4a+6b-2)}{2}$. This equals $\frac{D^{2}-2}{2}$ if and only if $a=1$ and $b=0$, hence $D \in |H|$.

For items (iii) and (iv), we consider the remaining cases of (ii) and those arising from item (2)(iii) of Theorem~\ref{teoprinc}:
\begin{itemize}
    \item[Case 1:] Item (ii) with $k=0$ is not possible since the degree is $-2$.
    \item[Case 2:] Item (ii) with $k=1$ is not possible since it corresponds to the intersection of four planes.
    \item[Case 3:] Item (ii) with $k=2$ gives the ACM curve $C_{1}$ of degree $6$ and genus $3$.
    \item[Case 4:] By Theorem~\ref{teoprinc}, item (2)(iii), for the pair $((1,1,1,1),(2,2,2,2))$, we obtain the resolution
    \[
    0 \to \oo_{\PPP}(-4)^{3} \to \oo_{\PPP}(-3)^{4} \to \ii_{D} \to 0,
    \]
    which corresponds to an ACM curve of degree $6$ and genus $3$.
\end{itemize}

For the last item, as in the previous propositions, complete intersections of surfaces of degrees $m,n \leq 3$ are not possible. Therefore, the only remaining cases are complete intersections of $X$ with a surface of degree $d \geq 4$, in which case the class is $dH$.\\
\end{proof}

\section{A generalization} \label{sec::4}
In this section, we present a generalization of the previous results to higher dimension.

\begin{defi}
Let $(\hat{a},\hat{b})$ be an admissible pair of degree $d$ and length $t$. A hypersurface
\[
X=\{F=0\}\subset \PP^{n}
\]
is called a \textit{weak determinantal hypersurface} of type $(\hat{a},\hat{b})$ if
\[
F=\det(S)
\]
for some matrix $S=(m_{ij})$, where $m_{ij}$ is a homogeneous polynomial of degree $b_{j}-a_{i}$. In the case $a_{i}\geq b_{j}$, we assume that $m_{ij}=0$.
\end{defi}

\begin{prop}
Let $X=V(F)\subseteq \PP^{n}$ be a hypersurface of degree $d$. If there exists an ACM closed subvariety
\[
Y\subseteq X
\]
of codimension one in $X$ that is not a complete intersection, then $X$ is a weak determinantal hypersurface of type $(\hat{a},\hat{b})$ for some admissible pair $(\hat{a},\hat{b})$ of degree $d$ and length $t$.
\end{prop}

\begin{proof}
Since there is a generalization of the Hilbert--Burch Theorem for subvarieties of codimension $2$ in $\PP^{n}$ (see \cite{har3}*{Prop. 8.7}), the proof is identical to the surface case in $\PPP$.\\
\end{proof}

As before, we obtain a classification theorem for ACM closed subvarieties of codimension one on a hypersurface of degree $d$ in $\PP^{n}$. The proof follows exactly the same arguments as in the surface case.

\begin{thm}
Let $X=V(F)$ be a hypersurface of degree $d$ in $\PP^{n}$, and let $Y\subseteq X$ be a closed subvariety of codimension one in $X$. Then $Y$ is ACM if and only if one of the following cases occurs:
\begin{enumerate}
    \item The hypersurface $X$ is not weak determinantal. In this case, $Y$ is a complete intersection.

    \item The hypersurface $X$ is weak determinantal of type $\{(\hat{a}_{i},\hat{b}_{i})\}_{i\in I}$, and only of these types. Then $Y$ is one of the following:
    \begin{enumerate}[label=(\roman*)]
        \item $Y$ is a complete intersection.

        \item If $Y$ is not a complete intersection and $F$ is a minimal generator of the ideal $\ii_{Y}$, then $Y$ has minimal free resolution
        \[
        0 \to \bigoplus_{j=1}^{t} \oo_{\PP^{n}}(-(b_{j}+k))
        \to \bigoplus_{i=1}^{t} \oo_{\PP^{n}}(-(a_{i}+k)) \oplus \oo_{\PP^{n}}(-d)
        \to \ii_{Y} \to 0,
        \]
        with $\hat{a}_{j}=(a_{1},\ldots,a_{t})$ and $\hat{b}_{j}=(b_{1},\ldots,b_{t})$ for some $j\in I$ and some $k\in\mathds{Z}$.

        \item If $Y$ is not a complete intersection and $F$ is not a minimal generator of the ideal $\ii_{Y}$, then $Y$ has minimal free resolution
        \[
        0 \to \bigoplus_{\substack{i=1\\ i\neq j_{0}}}^{t}
        \oo_{\PP^{n}}(-d+b_{j_{0}}-b_{i})
        \to \bigoplus_{i=1}^{t}
        \oo_{\PP^{n}}(-d+b_{j_{0}}-a_{i})
        \to \ii_{Y} \to 0,
        \]
        with
        \[
        \hat{a}_{j}=(a_{1}+k,\ldots,a_{t}+k), \quad
        \hat{b}_{j}=(b_{1}+k,\ldots,b_{j_{0}-1}+k,d,b_{j_{0}+1}+k,\ldots,b_{t}+k),
        \]
        for some $j\in I$, $j_{0}\in\{1,\ldots,t\}$, and some $k\in\mathds{Z}$.
    \end{enumerate}
\end{enumerate}
\end{thm}

\end{document}